\documentclass[reqno,10pt]{amsart}

\usepackage{latexsym,amssymb,amsthm,amsmath,amscd,a4wide}

\theoremstyle{plain}
\newtheorem{theorem}{Theorem}[section]
\newtheorem*{Theorem B}{Theorem B}
\newtheorem*{Theorem A}{Theorem A}
\newtheorem*{Theorem C}{Theorem C}
\newtheorem*{Theorem D}{Theorem D}

\newtheorem{lemma}{Lemma}[section]

\newtheorem{proposition}{Proposition}[section]

\newtheorem{example}{Example}[section]
\numberwithin{equation}{section}

\theoremstyle{remark}
\newtheorem{remark}{Remark}[section]

 \numberwithin{equation}{section}

\def\<{\left < }
\def\>{\right >}
\def\({\left ( }
\def\){\right )}
\def\o{\omega }
\def\lam{\lambda}

\def\e{\eqref}
\def\g{\gamma}

\def\tn{\tilde\nabla}

\def\n2{\left[{n\over2}\right]}

\usepackage[usenames]{color}

\begin{document}

\title[$\delta(2,n\hskip-.01in-\hskip-.01in 2)$-ideal Lagrangian submanifolds]{Classification of $\Large{\boldsymbol \delta} \Large{\boldsymbol (} \Large{\boldsymbol 2} \Large{\boldsymbol ,} \Large{\boldsymbol n} \hskip-.01in \Large{\boldsymbol -} \hskip-.01in \Large{\boldsymbol 2} \Large{\boldsymbol )}$-ideal Lagrangian submanifolds\\ in $\Large{\boldsymbol n}$-dimensional complex space forms}

\author[B.-Y. Chen]{Bang-Yen Chen}
\address{Chen: Department of Mathematics \\ Michigan State University \\ East Lansing, Michigan 48824--1027 \\ U.S.A.}
\email{bychen@math.msu.edu}

\author[F. Dillen]{Franki Dillen}

\author[J. Van der Veken]{Joeri Van der Veken}
\address{Dillen, Van der Veken: KU  Leuven \\ Departement Wiskunde \\ Celestijnenlaan 200 B, Box 2400 \\ BE-3001 Leuven \\ Belgium} 
\email{joeri.vanderveken@wis.kuleuven.be}

\author[L. Vrancken]{Luc Vrancken}
\address{Vrancken: LAMAV, ISTV2 \\Universit\'e de Valenciennes\\ Campus du Mont Houy\\ 59313 Valenciennes Cedex 9\\ France and KU  Leuven \\ Departement Wiskunde \\Celestijnenlaan 200 B, Box 2400 \\ BE-3001 Leuven \\ Belgium}  
\email{luc.vrancken@univ-valenciennes.fr}

\begin{abstract}   It was proven in \cite{cdvv}  that every Lagrangian submanifold
$M$ of a complex space form $\tilde M^{n}(4c)$ of constant holomorphic sectional
curvature $4c$ satisfies the following optimal inequality:
\begin{align*} 
\delta(2,n\hskip-.01in-\hskip-.01in 2) \leq \text{$ \frac{n^{2}(n-2)}{4(n-1)} $} H^2 +2(n-2) c,
\end{align*} 
where $H^{2}$ is the squared mean curvature and $\delta(2,n\hskip-.01in-\hskip-.01in 2)$ is a $\delta$-invariant on $M$. In this paper we classify Lagrangian submanifolds of complex space forms $\tilde M^{n}(4c)$, $n \geq 5$, which satisfy the equality case of this  improved inequality at every point.
\end{abstract}


\thanks{This research was supported by the Belgian Interuniversity Attraction Pole  P07/18 (Dygest) and project 3E160361 (Lagrangian and calibrated submanifolds) of the KU Leuven Research Fund}

\maketitle

\section{Introduction} 

Let $M$ be an $n$-dimensional Riemannian manifold and denote for all $p\in M$ and for all plane sections $\pi\subseteq T_pM$, the sectional curvature of $M$ associated with $\pi$ by $K(\pi)$. If $L$ is an $r$-dimensional subspace of $T_pM$ with $2 \leq r \leq n$ and $\{e_1,\ldots,e_r\}$ is an orthonormal basis of $L$, the scalar curvature of $L$ is defined by
\begin{align}\label{1.2}
\tau(L)=\sum_{{\alpha,\beta=1}\atop{\alpha<\beta}}^r K(e_\alpha\wedge e_\beta).
\end{align}
It is easily checked that this definition does not depend on the chosen orthonormal basis of $L$. In particular, the scalar curvature $\tau$ of $M$ at $p$ is defined to be $\tau(p) = \tau(T_pM)$.

For given integers $n\geq 3$ and  $k\geq 1$, we denote by $\mathcal S(n,k)$ the finite set  consisting of all $k$-tuples $(n_1,\ldots,n_k)$ of integers  satisfying  $2 \leq n_1 \leq \cdots \leq
n_k \leq n-1$ and $n_1+\cdots+n_k\leq n.$ Denote the union $\bigcup_{k\geq 1}\mathcal S(n,k)$ by ${\mathcal S}(n)$.
For each $(n_1,\ldots,n_k)\in \mathcal S(n)$, the first author introduced in \cite{c00a} the Riemannian invariant $\delta{(n_1,\ldots,n_k)}$ defined  by
\begin{align}\label{1.3} 
\delta(n_1,\ldots,n_k)(p)=\tau(p)- \inf\{\tau(L_1)+\cdots+\tau(L_k)\}
\end{align} 
for any $p\in M^{n}$, where $L_1,\ldots,L_k$ run over all $k$-tuples of mutually orthogonal subspaces of $T_pM^{n}$ such that  $\dim L_j=n_j$ for $j=1,\ldots,k$. 

For any submanifold of a real space form of constant sectional curvature~$c$, we have the following sharp general inequality relating intrinsic data of the submanifold (the $\delta$-invariant) with extrinsic data of the immersion (the mean curvature). We refer to \cite{c00a,book} for more details.

\begin{theorem}\label{T:1.1} Let $M$ be an $n$-dimensional submanifold of a real space form of constant sectional curvature $c$. Then for each $k$-tuple  $(n_1,\ldots,n_k)\in\mathcal S(n)$ and at any point $p \in M$, the following inequality holds:
\begin{align*} 
\delta(n_1,\ldots,n_k) \leq \frac{n^2(n+k-1-\sum_{j=1}^k n_j)}{2(n+k-\sum_{j=1}^k n_j)} H^2 + b(n_1,\ldots,n_k)c,
\end{align*} 
where $H^2$ is the squared mean curvature of $M$ at $p$ and $b(n_1,\ldots,n_k)$ is defined by 
\begin{align*} b(n_1,\ldots,n_k)=\frac{n(n-1)}{2}-\sum_{j=1}^k \frac{n_j(n_j-1)}{2}.
\end{align*}
\end{theorem}
 
The same inequality holds for Lagrangian submanifolds of a complex space form $\tilde M^n(4c)$, but it is not optimal in that case. Recall that a submanifold of a K\"ahler manifold is called \emph{Lagrangian} if the almost complex structure $J$ induces an isomorphism between the tangent space and the normal space at every point or, equivalently, if the K\"ahler 2-form restricted to the submanifold vanishes. 

An optimal result for Lagrangian submanifolds was obtained in \cite{cdvv}, where a distinction needed to be made between the cases $n_1+\ldots+n_k < n$ and $n_1+\ldots+n_k = n$. In particular, we obtained the following results.

\begin{theorem} \label{T:1.2}
Let $M$ be a Lagrangian submanifold of a complex space form
$\tilde M^n(4c)$. Then for each $k$-tuple $(n_1,\ldots,n_k)\in\mathcal S(n)$ with $n_1+\ldots+n_k < n$, and at any point of $M^n$, the following inequality holds:
\begin{equation*}
\delta(n_1,\ldots,n_k)
\leq \frac{n^2 \left(n -\sum_{j=1}^k n_j + 3k - 1 - 6\sum_{j=1}^k \frac{1}{2+n_j} \right)}
{2 \left(n -\sum_{j=1}^k n_j + 3k + 2 - 6\sum_{j=1}^k \frac{1}{2+n_j} \right)} H^2
+ b(n_1,\ldots,n_k)c,
\end{equation*}
where $b(n_1,\ldots,n_k)$ is as in Theorem \ref{T:1.1}.
\end{theorem}

\begin{theorem} \label{T:1.3}
Let $M$ be a Lagrangian submanifold of a complex space form
$\tilde M^n(4c)$. Then for each $k$-tuple $(n_1,\ldots,n_k)\in\mathcal S(n)$ with $n_1+\ldots+n_k = n$, and at any point of $M^n$, the following inequality holds:
$$\delta(n_1,\ldots,n_k)
\leq \frac{n^2\left(k-1-2\sum_{j=2}^k\frac{1}{n_j+2}\right)}{2\left(k-2\sum_{j=2}^k\frac{1}{n_j+2}\right)}H^2
+b(n_1,\ldots,n_k)c,
$$
where $b(n_1,\ldots,n_k)$ is as in Theorem \ref{T:1.1}.
\end{theorem}

In both cases, a (different) full description of the second fundamental form of those submanifolds realizing equality in the inequality at any of their points is also given in \cite{cdvv}. We call such a Lagrangian submanifold \emph{$\delta(n_1,\ldots,n_k)$-ideal}. Since the mean curvature is a measure for the tension a submanifold experiences from its shape in the ambient space, the submanifolds are shaped ideally in the sense that they experience the least amount of tension, given their intrinsic geometry. The full descriptions of the second fundamental forms would require us to introduce a lot of new notation, so we will restrict to the case treated in this paper, which is a special case of Theorem \ref{T:1.3}.

\begin{theorem}\label{T:1.4} For a Lagrangian submanifold $M$ of a complex space form $\tilde M^{n}(4c)$ with $n\geq 5$,  we have
 \begin{equation} \label{L23} 
\delta(2,n \hskip-.01in-\hskip-.01in 2) \leq \text{$ \frac{n^{2}(n-2)}{4(n-1)} $} H^2 +2(n-2) c.
\end{equation}

If the equality sign in \e{L23} holds at a point $p$, then there exists an orthonormal basis $\{e_{1},\ldots,e_{n}\}$ of $T_{p}M$ such that the components of the second fundamental form, $h_{AB}^C = \langle h(e_A,e_B),Je_C\rangle$, satisfy
\begin{align} 
& \label{1.9}  h^{k}_{11} = h^{k}_{22} = h^{k}_{33} + \cdots + h^{k}_{nn} = 0 && \mbox{for } k \geq 3, \\
& \label{1.10} h^{i}_{11}+ h^{i}_{22}= n h^{i}_{33} = \cdots = nh^{i}_{nn} && \mbox{for } i \in \{1,2\}, \\
& \label{1.11} h^{1}_{k\ell} = h^{2}_{k\ell} = h^{k}_{12} = 0 && \mbox{for } k, \ell \geq 3, \ k \neq \ell.
\end{align}
\end{theorem}

The purpose of this paper is to classify $\delta(2,n\hskip-.01in-\hskip-.01in 2)$-ideal Lagrangian submanifolds in complex space forms for $n \geq 5$. Remark that the latter condition is necessary: for $\delta(2,2)$-ideal Lagrangians in $\tilde M^4(4c)$, the description of the second fundamental form is different (cfr. \cite{cdvv}).

The paper is organised as follows. Section 2 contains some preliminaries on submanifold theory and in particular on Lagrangian submanifolds of complex space forms. In Section 3, the second fundamental form of $\delta(2,n-2)$-ideal Lagrangian submanifolds of complex space forms with complex dimension $n\geq 5$ is determined, along with some additional information. It turns out that, apart from the minimal case, case (I), there are two other cases to consider: case (II) is completely solved in Section 4, by reducing it to a special case of a family of Lagrangians studied in \cite{cvv}. Case (III) is more involved and is treated in Section 5. Section 6 contains the final conclusions of the paper.

\section{Preliminaries}

\subsection{Basic formulas}
If $\tilde M^n(4c)$ is a complete simply connected K\"ahler $n$-manifold  with constant holomorphic sectional curvature $4c$, then $\tilde M^n(4c)$ is holomorphically isometric to the complex Euclidean $n$-space ${\bf C}^n$, the complex projective $n$-space $CP^n(4c)$, or the complex hyperbolic $n$-space $CH^n(-4c)$ according to $c=0$, $c>0$ or $c<0$ respectively. These manifolds are known as \emph{complex space forms}.

Let $M$ be a Lagrangian submanifold of  $\tilde M^n(4c)$.
Denote the Levi-Civita connections of $M$ and $\tilde M^n(4c)$ by $\nabla$ and $\tilde \nabla$, respectively.
The formulas of Gauss and Weingarten are given respectively by (cf. \cite{book})
\begin{align}\label{2.1} 
&\tn_X Y = \nabla_X Y + h(X,Y), \quad \tn_X \xi = -A_\xi X + \nabla^{\perp}_X \xi
\end{align}
for tangent vector fields $X$ and $Y$ and normal vector fields  $\xi$, where $h$ is the second fundamental form, $A$ is the shape operator and $\nabla^{\perp}$ is the normal connection. The second fundamental form  and the shape operator are related by $\<h(X,Y),\xi\> = \<A_\xi X,Y\>$. The mean curvature vector field of $M$ is defined by $H=(\hbox{trace}\,h)/n$ and the \emph{squared mean curvature} is given by $H^2=\<H,H\>$.

For a Lagrangian submanifold, we have (cf. \cite{book,CO})
\begin{align}\label{2.3} &\nabla^{\perp}_X JY = J \nabla_X Y,
\\\label{2.4} &A_{JX} Y = -J h(X,Y)=A_{JY}X
\end{align}
for all tangent vector fields $X$ and $Y$. Formula \e{2.4} implies in particular that the so-called cubic form $(X,Y,Z) \mapsto \<h(X,Y),JZ\>$ is totally symmetric. 
For an orthonormal basis $\{e_1,\ldots,e_n\}$ of $T_pM$, we put 
\begin{align}\label{2.7} 
h^C_{AB}=\<h(e_A,e_B),Je_C\>.
\end{align}

The equations of Gauss and Codazzi are given respectively by
\begin{align} \label{Gauss} 
& \<R(X,Y)Z,W\> = c(\<X,W\>\<Y,Z\>-\<X,Z\>\<Y,W\>)\\&
\notag \hskip1.0in +  \<h(X,W),h(Y,Z)\> - \<h(X,Z), h(Y,W)\> ,
\\ &\label{Codazzi} (\nabla_{X} h)(Y,Z) = (\nabla_{Y} h)(X,Z),\end{align}
where $R$ is the curvature tensor of $M$ and $\nabla h$ is defined by 
\begin{align}\label{2.6}
(\nabla_{X} h)(Y,Z) = \nabla^{\perp}_X h(Y,Z) - h(\nabla_X Y,Z) - h(Y,\nabla_X Z).
\end{align}

\subsection{Horizontal lifts of Lagrangian submanifolds}

We recall the link between Legendre submanifolds and Lagrangian submanifolds (cf. \cite{book,rec}).
\vskip.04in

\noindent {\it Case} (i): $CP^n(4)$.  Consider the Hopf fibration $\pi :S^{2n+1}\to CP^n(4)$, where $S^{2n+1}$ is the unit sphere in $\mathbf C^{n+1}$. For a given point $u\in S^{2n+1}$, the horizontal space at $u$ is the orthogonal complement of $i u, \, i=\sqrt{-1},$ with respect to the metric  on $S^{2n+1}$ induced from the metric on ${\bf C}^{n+1}$.   
Let $L : M \to CP^n(4)$ be a Lagrangian isometric immersion. Then there is a covering map
$\tau: \hat M \to M$ and a horizontal  immersion $\tilde L :\hat M \to S^{2n+1}$ such that
$L\circ \tau=\pi \circ \tilde L$.  Thus each Lagrangian immersion can be lifted locally (or globally if $M$ is simply connected) to a Legendre immersion of the same Riemannian manifold. In particular, a minimal Lagrangian submanifold of $CP^{n}(4)$ is lifted to a minimal Legendre submanifold of the Sasakian manifold $S^{2n+1}$.

Conversely, suppose that $\tilde L: M \to S^{2n+1}$ is a Legendre isometric immersion. Then $L =\pi\circ \tilde L:  M\to  CP^n(4)$ is a Lagrangian isometric immersion.  Under this correspondence the second fundamental forms $h^{\tilde L}$ and $h^L$ of $\tilde L$ and $L$ satisfy $\pi_*h^{\tilde L}=h^L$. Moreover, $h^{\tilde L}$ is horizontal with respect to $\pi$.  
\vskip.04in

\noindent {\it Case} (ii): $CH^n(-4)$.  We consider  the complex number space  ${\bf C}^{n+1}_1$ equipped with the pseudo-Euclidean metric
$g_0=-dz_1d\bar z_1 + dz_2d\bar z_2 + \ldots + dz_{n+1}d\bar z_{n+1}$
and look at
$$H^{2n+1}_{1}=\{z\in {\bf C}^{n+1}_1 \ | \ \<z,z\>=-1\}$$ 
with the canonical Sasakian structure, where $\<\;\,,\;\>$ is the induced inner product from $g_0$. In particular, $H_1^1=\{\lambda\in {\bf  C} \ | \ \lambda\bar\lambda=1\}.$ Then there is an $H^1_1$-action on $H_1^{2n+1}$, given by $z\mapsto \lambda z$, and at each point $z\in H^{2n+1}_1$, the vector $\xi=i z$ is tangent to the flow of the action. Since the metric $g_0$ is Hermitian, we have $\<\xi,\xi\>=-1$.  The quotient space $H^{2n+1}_1/\sim$, under the identification induced from the action, is the complex
hyperbolic space $CH^n(-4)$ with constant holomorphic sectional curvature $-4$ whose complex structure $J$ is  induced from the complex structure on ${\bf C}^{n+1}_1$ via the Hopf fibration $\pi :H^{2n+1}_1\to  CH^n(-4).$

Just like in case (i), if $L: M \to CH^n(-4)$ is a Lagrangian immersion, then there is an isometric covering map $\tau: \hat M \to M$ and a Legendre immersion $\tilde L: \hat M \to H_1^{2n+1}$ such that $L \circ \tau=\pi\circ \tilde L$. Thus every Lagrangian immersion into $CH^n(-4)$ an be lifted locally (or globally if $M$ is simply connected) to a Legendre immersion into $H^{2n+1}_1$. In particular,   minimal Lagrangian submanifolds of $CH^{n}(-4)$ are lifted to minimal Legendre submanifolds of $H^{2n+1}_{1}$. Conversely, if $\tilde L:\hat M \to H_1^{2n+1}$ is a Legendre immersion, then $L =\pi\circ \tilde L:  M\to CH^n(-4)$ is a Lagrangian immersion.  Under this correspondence the second fundamental forms $h^f$ and $h^L$ are related by $\pi_*h^{\tilde L}=h^L$. Also, $h^{\tilde L}$ is horizontal with respect to $\pi$. 

Let $h$ be the second fundamental form of $M$ in $S^{2n+1}$, respectively $H^{2n+1}_1$. Since $S^{2n+1}$ and $H^{2n+1}_1$ are totally umbilical  with mean curvature $1$ in ${\bf C}^{n+1}$, respectively ${\bf C}^{n+1}_1$, we have  
\begin{align}\label{2.9} D_XY=\nabla_XY+h(X,Y)-\varepsilon \tilde L,\end{align} 
where $\varepsilon=1$ if the ambient space is ${\bf C}^{n+1}$ and $\varepsilon=-1$ if it is ${\bf C}^{n+1}_1$ and $D$ denotes the Levi-Civita connection of ${\bf C}^{n+1}$, respectively ${\bf C}^{n+1}_1$.

\section{The second fundamental form of $\delta(2,n-2)$-ideal Lagrangian submanifolds}

In this section, we prove two lemmas. The first one, Lemma \ref{L:3.2}, describes the second fundamental form of a $\delta(2,n-2)$-ideal Lagrangian submanifold of a complex space form pointwise and follows from Theorem \ref{T:1.4}. The second one, Lemma \ref{L:3.3}, describes the second fundamental form in terms of a local orthonormal frame. 

\begin{lemma}\label{L:3.2} Let $M$ be a Lagrangian submanifold of a complex space form $\tilde M^{n}(4c)$, $n\geq 5$,  satisfying the equality case of \e{L23} at a point $p \in M$. Then there exist an orthonormal basis $\{e_{1},\ldots,e_{n}\}$ of $T_pM$ and real numbers $\g,\lambda,\mu$ and $h^{k}_{ij}$ $(i,j,k \geq 3)$, such that 
\begin{equation}\begin{aligned}\label{3.4} 
& h(e_1,e_1)=\g Je_1, \ \ 
  h(e_1,e_2)=(n\lambda-\g)Je_2, \\ 
& h(e_2,e_2)=(n\lambda-\g)Je_1 +n \mu Je_2, \\
& h(e_{1},e_{i})=\lambda Je_{i}, \ \ h(e_{2},e_{i})=\mu Je_{i}, \\
& h(e_{i},e_{j})= \delta_{ij}(\lambda Je_{1}+\mu J e_{2})+\sum_{k=3}^{n} h^{k}_{ij}Je_{k}
\end{aligned}\end{equation} 
for $i,j \geq 3$. The numbers $h^{k}_{ij}$ are symmetric in the three indices and satisfy 
$h^{k}_{33} + \ldots + h^k_{nn}=0$ for any $k \geq 3$. Moreover,
\begin{align}
& \gamma \geq 0, \ \gamma \geq \frac {2n}{3} \lambda, \label{3.4i} \\
& \mbox{if } \g=0, \mbox{ then also } \lambda=\mu=0, \\
& \mbox{if } \g>0, \mbox{ then also } \gamma > \frac n2 \lambda.
\end{align}
\end{lemma}

\begin{proof}
Choose an orthonormal basis $\{e_1,\ldots,e_n\}$ of $T_pM$ such that \eqref{1.9}--\eqref{1.11} hold. This implies that
\begin{equation*}
\begin{aligned} 
& h(e_1,e_1) = h_{11}^1 Je_1 + h_{11}^2 Je_2, \\ 
& h(e_1,e_2) = h_{11}^2 Je_1 + h_{22}^1 Je_2, \\ 
& h(e_2,e_2) = h_{22}^1 Je_1 + h_{22}^2 Je_2, \\
& h(e_1,e_k) = h_{33}^1 Je_k, \quad h(e_2,e_k) = h_{33}^2 Je_k, \\
& h(e_k,e_{\ell} )= \delta_{k\ell}(h_{33}^1 Je_1 + h_{33}^2 Je_2) + \sum_{m=3}^{n} h_{k\ell}^m Je_m,
\end{aligned}
\end{equation*} 
with
\begin{equation*}
\begin{aligned} 
& h_{11}^1 + h_{22}^1 = nh_{33}^1 \ (= nh_{44}^1 = \ldots = nh_{nn}^1), \\ 
& h_{11}^2 + h_{22}^2 = nh_{33}^2 \ (= nh_{44}^2 = \ldots = nh_{nn}^2), \\
& h_{33}^k + \ldots + h_{nn}^k = 0 \mbox{ for } k \geq 3.
\end{aligned}
\end{equation*} 

Remark that the conditions \eqref{1.9}--\eqref{1.11} remain true for any choice of orthonormal basis in $\mathrm{span}\{e_1,e_2\}$. In particular, we can assume that the following function, defined on a compact set, attains its global maximum in $e_1$:
$$ \phi: \{ u \in \mathrm{span}\{e_1,e_2\} \ | \ \|u\|=1 \} \to \mathbb R : u \mapsto \langle h(u,u),Ju \rangle. $$
This implies that the function $F : \mathbb R \to \mathbb R: \theta \mapsto \phi((\cos\theta) e_1 + (\sin\theta)e_2)$ attains a maximum at $\theta=0$. Computing the first and second derivatives of $F$ gives respectively
$h_{11}^2 = 0$ and $h_{11}^1 \geq 2 h_{22}^1$. Since $\phi(-e_1)=-\phi(e_1)$ and $\phi$ attains its maximum at $e_1$, we have $\phi(e_1)=h_{11}^1 \geq 0$. Moreover, if $h_{11}^1 = 0$, then $\phi$ vanishes identically, which implies that also $h_{22}^1 = h_{22}^2 = 0$. Finally, if $h_{11}^1 > 0$, it is easy to see that $h_{11}^1 > h_{22}^1$.

We now obtain the result by putting $\gamma=h_{11}^1$, $\lambda=h_{33}^1$ and $\mu=h_{33}^2$.
\end{proof}

Remark that, under the assumptions of Lemma \ref{L:3.2}, the mean curvature vector at the point $p$ is given by
\begin{align} \label{expressionH}
& H(p) = \frac{2(n-1)}{n}(\lambda Je_1 + \mu Je_2). 
\end{align}

It is not clear whether the orthonormal bases given by Lemma \ref{L:3.2} at every point of a $\delta(2,n-2)$-ideal Lagrangian submanifold of a complex space form can be pasted together to form a differentiable orthonormal frame. However, we have the following local result.

\begin{lemma}\label{L:3.3}
Let $M$ be a $\delta(2,n-2)$-ideal Lagrangian submanifold of a complex space form $\tilde M^{n}(4c)$, $n\geq 5$. Then there exists an open and dense subset $V \subseteq  M$ such that every point of $V$ has a neighborhood in which one of the following holds.
\begin{itemize}
\item[(I)] $H = 0$.
\item[(II)] There exists a differentiable orthonormal frame $\{E_1,\ldots,E_n\}$ such that the second fundamental form satisfies
\begin{equation} \label{h_case(II)}
\begin{aligned}
& h(E_1,E_1)=(n-1)\lambda JE_1, \ h(E_1,E_i) = \lambda JE_i, \\
& h(E_i,E_j) = \delta_{ij} \lambda JE_1 + \sum_{k=2}^n h_{ij}^k JE_k
\end{aligned}
\end{equation}
for $i,j \geq 2$, where $\lambda$ and $h_{ij}^k$ are differentiable functions, the latter being symmetric in the three indices and satisfying $h_{22}^k + \ldots + h_{nn}^k = 0$ for all $k \geq 2$.
\item[(III)] There exists a differentiable orthonormal frame $\{E_1,\ldots,E_n\}$ such that the second fundamental form satisfies
\begin{equation} \label{h_case(III)}
\begin{aligned}
& h(E_1,E_1) = \g JE_1, \ \ 
  h(E_1,E_2) = (n\lambda-\g)JE_2, \\
& h(E_2,E_2) = (n\lambda-\g)JE_1 + n \mu JE_2, \\
& h(E_1,E_i) = \lambda JE_i, \ \ 
  h(E_2,E_i) = \mu JE_i, \\
& h(E_i,E_j) = \delta_{ij}(\lambda JE_1 + \mu J E_2)+\sum_{k=3}^{n} h^{k}_{ij}JE_k
\end{aligned}
\end{equation}
for $i,j \geq 3$, where $\gamma$, $\lambda$, $\mu$ and $h^{k}_{ij}$ are differentiable functions, the latter being symmetric in the three indices, satisfying $\gamma > 0$, $\gamma > 2n\lambda/3$ and $h^{k}_{33} + \ldots + h^k_{nn}=0$ for all $k \geq 3$. Moreover, at every point, $\lambda \neq 0$ or $\mu \neq 0$, and also $\mu \neq 0$ or $\gamma \neq (n-1)\lambda$.
\end{itemize}
\end{lemma}

\begin{proof}
Define $V_1 = \{ p \in M \ | \ H(p) \neq 0 \}$ and $V_2 = \{ p \in M \ | \ H(p)=0 \}^{\mbox{int}}$, where the superscript ``int'' denotes the interior. Clearly, all points in $V_2$ satisfy case (I). 

On $V_1$, we consider the $(1,1)$-tensor field 
\begin{equation} \label{defK}
K : \mathcal D \to \mathcal D : X \mapsto \pi_{\mathcal D} J h(JH,X), 
\end{equation}
where $\mathcal D$ is the orthogonal complement of $\mbox{span}\{JH\}$ in the tangent space to $M$ and $\pi_{\mathcal D}$ is the orthogonal projection onto $\mathcal D$ at every point. Define further
\begin{align*}
& V_{11} = \{ p \in V_1 \ | \ h(JH(p),JH(p)) \mbox{ is no multiple of } H(p) \mbox{ or } K_p \mbox{ is no multiple of } \mathrm{id}_{{\mathcal D}_p} \}, \\
& V_{12} = \{ p \in V_1 \ | \ h(JH(p),JH(p)) \mbox{ is a multiple of } H(p) \mbox{ and } K_p \mbox{ is a multiple of } \mathrm{id}_{{\mathcal D}_p} \}^{\mbox{int}}.
\end{align*}
If $p \in V_{12}$ and $\{e_1,\ldots,e_n\}$ is an orthonormal basis of $T_pM$ as in Lemma \ref{L:3.2}, it follows from the definition of $V_{12}$ and a straightforward computation using \eqref{3.4}, \eqref{expressionH} and \eqref{defK} that $\gamma=(n-1)\lambda$ and $\mu=0$. In particular, $e_1$ lies in the direction of $H(p)$. This means that we can extend $\{e_1,\ldots,e_n\}$ to an orthonormal frame $\{E_1,\ldots,E_n\}$ on $V_{12}$, where $E_1$ lies in the direction of $H$ at every point, and we are in case (II).

Finally, let $p \in V_{11}$ and consider an orthonormal basis $\{e_1,\ldots,e_n\}$ of $T_pM$ as in Lemma~\ref{L:3.2}. Putting $(\mathcal D_1)_p = \mathrm{span}\{e_1,e_2\}$ and $(\mathcal D_2)_p = \mathrm{span}\{e_3,\ldots,e_n\}$, we shall now prove that $\mathcal D_1$ and $\mathcal D_2$ are differentiable distributions on $V_{11}$. If $h(JH(p),JH(p))$ is not parallel with $H(p)$, then the same holds in a neighborhood of $p$ and it follows from \eqref{3.4} and \eqref{expressionH} that $\mathcal D_1 = \mathrm{span}\{JH,Jh(JH,JH)\}$ in this neighborhood. Hence, $\mathcal D_1$ is differentiable in this neighborhood. If, on the other hand, $h(JH(p),JH(p))$ and $H(p)$ are parallel, then, by the definition of $V_{11}$, we have that $K_p$ is not a multiple of $\mathrm{id}_{\mathcal D_p}$ and it follows from \eqref{3.4} and \eqref{defK} that the matrix of $K_p$ with respect to the orthonormal basis $\left\{ (\mu e_1 - \lambda e_2)/\sqrt{\lambda^2+\mu^2},e_3,\ldots,e_n \right\}$ of $\mathcal D_p$, is given by
$$ \frac{2(n-1)}{n} \left( \begin{array}{cccc}
\alpha & & & \\ & \lambda^2 + \mu^2 & & \\ & & \ddots & \\ & & & \lambda^2 + \mu^2
\end{array} \right)$$
for some real number $\alpha \neq \lambda^2 + \mu^2$. The same holds in a neighborhood of $p$ and hence there is a well-defined one-dimensional eigendistribution of the tensor field $K$, say $\mathrm{span}\{X_0\}$. Since $K$ is differentiable, the vector field $X_0$ can be chosen to be differentiable and hence $\mathcal D_1 = \mathrm{span}\{JH,X\}$ is differentiable in a neighborhood of $p$. In both cases, $\mathcal D_2$ is differentiable since it is the orthogonal complement of $\mathcal D_1$ in $TM$.

Let $\{X_1,X_2\}$ be differentiable orthonormal vector fields on $V_{11}$ spanning $\mathcal D_1$ at every point and $\{E_3,\ldots,E_n\}$ differentiable orthonormal vector fields on $V_{11}$ spanning $\mathcal D_2$ at every point. In order to obtain case (III) of the lemma, we have to find a differentiable function $\theta$ on $V_{11}$ such that $E_1 = (\cos\theta) X_1 + (\sin\theta) X_2$ maximizes
$ \phi: \{ X \in \mathcal D_1 \ | \ \|X\|=1 \} \to \mathbb R : X \mapsto \langle h(X,X),JX \rangle $
at every point. This implies that $E_2 = -(\sin\theta) X_1 + (\cos\theta) X_2$ satisfies 
\begin{equation}\label{3.5bis}
\langle h(E_1,E_1),JE_2 \rangle = 0.
\end{equation} 
The latter equation has in general several differentiable solutions for $\theta$. However, since we want $E_1 = (\cos\theta) X_1 + (\sin\theta) X_2$ to maximize $\phi$, we have to restrict to points for which the number of solutions, say in $[0,2\pi)$, does not change in a neighborhood to guarantee differentiability of $\theta$. If we define $V_{111}$ as the set of those points in $V_{11}$ for which the number of solutions for $\theta$ of \eqref{3.5bis} in $[0,2\pi)$ does not change in a neighborhood of the point, we can construct an orthonormal frame on $V_{111}$ satisfying \eqref{h_case(III)} as explained above. Remark that $\gamma > 0$ and $\gamma > 2n\lambda/3$ follow from the last sentence of Lemma~\ref{L:3.2} and the fact that $H$ is nowhere vanishing on $V_{111}$. Moreover, the fact that $\lambda \neq 0$ or $\mu \neq 0$ also follows from the non-vanishing of $H$ and the fact that $\mu \neq 0$ or $\gamma \neq (n-1)\lambda$ follows from the definition of $V_{11}$ and the computation which led to case (II) above.

As a conclusion, the subset $V \subseteq M$ we are looking for is the disjoint union
$$ V = V_{111} \cup V_{12} \cup V_2, $$
which is open and dense in $M$ by construction. 
\end{proof}

We will proceed with the classification as follows. In Section 4, we give a classification in case (II), based on results in \cite{cvv}. In Section 5, we give a classification in case (III) and, finally, Section 6 contains the overall conclusions. We will not elaborate on case (I) in general, however, we remark the following.

\begin{remark} \label{RemarkMinimal}
If $M$ is a minimal $\delta(2,n-2)$-ideal Lagrangian submanifold of a complex space form $\tilde M^{n}(4c)$, $n\geq 5$, for which the orthonormal bases given in Lemma \ref{L:3.2} can be pasted together to form a differentiable orthonormal frame $\{E_1,\ldots,E_n\}$, then the second fundemental form is given by 
\begin{equation}
\begin{aligned}
\notag & h(E_1,E_1)=\g JE_1,\; h(E_1,E_2)=-\g JE_2, \; h(E_2,E_2)=-\g JE_1,
\\& h(E_{1},E_{i})= h(E_{2},E_{i})=0,
\; h(E_{i},E_{j})=\sum_{k=3}^{n} h^{k}_{ij}JE_{k}
\end{aligned}
\end{equation} 
for $i,j,k\geq 3$ and some functions $\g$, $\lambda$, $\mu$ and $h^{k}_{ij}$, satisfying $h^{k}_{33}+\ldots h_{nn}^k=0$ for every $k \geq 3$. If $\gamma = 0$, the Lagrangian submanifold is minimal $\delta(n-2)$-ideal. If $\gamma > 0$, a long argument, very similar to the one we will give in Section \ref{sec5.1}, can be used to prove that there are three possibilities: the Lagrangian submanifold is either minimal $\delta(2)$-ideal, minimal $\delta(2,k)$-ideal for some $k$ satisfying $2\leq k< n-2$ or it is a direct product of a minimal $\delta(2)$-ideal Lagrangian surface in $\mathbf C^2$ and a minimal $\delta(n-2)$-ideal submanifold of $\mathbf C^{n-2}$. The latter case only occurs for $c=0$. The family of minimal $\delta(2)$-ideal Lagrangians is too large to classify. On the other hand, minimal $\delta(2,2)$-ideal Lagrangians in dimension $5$ were classified in \cite{cpw}.
\end{remark}

\section{Classification in case (II) of Lemma \ref{L:3.3}}

Let $M$ be a Lagrangian submanifold of a complex space form $\tilde M^{n}(4c)$, $n\geq 5$, satisfying case (II) of Lemma \ref{L:3.3}. It was proven in \cite{cvv} that such a submanifold is a warped product $I \times_f N$ of an open interval $I$ and an $(n-1)$-dimensional factor $N$. Moreover, $E_1$ is tangent to $I$ and the Lagrangian immersion is constructed from a curve depending on a parameter $t \in I$, determined by a system of ODEs and a Lagrangian immersion of the manifold $N$, for which the components of the second fundamental form are, up to a factor depending on $t$, equal to the corresponding components of $h$.

Combining this result with Lemma \ref{L:3.2} yields that there exists an orthonormal basis $\{e_2,\ldots,e_n\}$ for every tangent space to $N$ such that the components of the second fundamental form $\tilde h$ of the Lagrangian immersion of $N$ satisfy $\tilde h^2_{k \ell} = 0$ for all $k, \ell \geq 2$ and $\tilde h_{22}^k+\ldots+\tilde h_{nn}^k=0$ for all $k \geq 2$. This means exactly that the Lagrangian immersion is $\delta(n-2)$-ideal and minimal. Hence, we obtain the following results (remark the slight difference in notation compared to \cite{cvv}).

\begin{proposition} \label{P:4.1} 
Let $M$ be a $\delta(2,n-2)$-ideal Lagrangian submanifold of the complex Euclidean space ${\bf C}^{n}$ $(n\geq 5)$ whose second fundamental form is given by case $\mathrm{(II)}$ of Lemma \ref{L:3.3}. Then $M$ is locally congruent to the image of
\begin{equation}\label{4.26} 
L(t, u_{2},\ldots, u_{n})=\frac{e^{i\theta}}{\varphi+i\lambda}\Phi( u_{2},\ldots,u_{n}),  
\end{equation} 
where $\theta$, $\varphi$ and $\lambda$ are functions of $t$ only, satisfying
\begin{equation} \label{systemODE}
\lambda' = (n-3) \lambda \varphi, \quad \varphi' = - \varphi^2 - (n-2)\lambda^2, \quad \theta' = (n-1) \lambda
\end{equation}
and $\Phi$ is a Legendre immersion into $S^{2n-1}(1)\subset {\bf C}^n$ whose composition with the Hopf fibration is a minimal $\delta(n-2)$-ideal Lagrangian immersion into $CP^{n-1}(4)$.
\end{proposition}

Remark that the system \eqref{systemODE} allows us to express all three unknown functions in terms of $\lambda$. Recall that $\lambda > 0$. It follows from the equations that $\lambda^{\frac{2}{n-3}}(\lambda^2 + \varphi^2)$ is a positive constant, say $r^2$ for some $r>0$. Then $\varphi = \pm \sqrt{r^2 \lambda^{-\frac{2}{n-3}}-\lambda^2}$. 
After replacing $E_1$ by $-E_1$ if necessary, we may assume that $\varphi > 0$ and thus
\begin{equation} \label{phi}
\varphi = \sqrt{\frac{1}{c^2 \, \lambda^{\frac{2}{n-3}}}-\lambda^2},
\end{equation}
where we have put $c = 1/r$. Since
$$ \frac{d\theta}{d\lambda} = \frac{(n-1) \lambda}{(n-3) \lambda \varphi} = \frac{n-1}{n-3} \, \frac{1}{\varphi}, $$
direct integration using \eqref{phi} yields
\begin{equation} \label{theta}
\theta = \frac{n-1}{n-2}\arcsin\left(c \, \lambda^{\frac{n-2}{n-3}} \right). 
\end{equation}
After a reparametrization $t \mapsto \lambda(t)$, the coefficient in front of $\Phi$ in \eqref{4.26} is completely determined by \eqref{phi} and \eqref{theta}.

\begin{proposition}\label{P:4.2} 
Let $M$ be a $\delta(2,n-2)$-ideal Lagrangian submanifold of the complex projective space  $CP^{n}(4)$ $(n\geq 5)$ whose second fundamental form is given by case $\mathrm{(II)}$ of Lemma \ref{L:3.3}. Then $M$ is locally congruent to the image of $\pi\circ L$, where $\pi:S^{2n+1}(1)\to CP^{n}(4)$ is the Hopf fibration and
\begin{equation}\label{4.43}  
L(t, u_{2},\ldots,u_{n})=\(\frac{e^{i \theta}\Phi(u_{2},\ldots,u_{n})}{\sqrt{1+\lambda^2+\varphi^2}}, \frac{ e^{i(n-2)\theta}(i\lambda-\varphi)}{\sqrt{1+\lambda^2+\varphi^2}}\),  
\end{equation}
where $\theta$, $\varphi$ and $\lambda$ are functions of $t$ only, satisfying
\begin{equation} \label{4.44} 
\lambda'=(n-3)\lambda\varphi, \quad \varphi'=-1-\varphi^2-(n-2)\lambda^2, \quad \theta'=\lambda
\end{equation}
and $\Phi$ is a Legendre immersion into $S^{2n-1}(1) \subset {\bf C}^n$ whose composition with the Hopf fibration is a minimal $\delta(n-2)$-ideal Lagrangian immersion into $CP^{n-1}(4)$.
\end{proposition}

\begin{proposition}\label{P:4.3} 
Let $M$ be a $\delta(2,n-2)$-ideal Lagrangian submanifold of the complex hyperbolic space $CH^{n}(-4)$ $(n\geq 5)$ whose second fundamental form is given by case $\mathrm{(II)}$ of Lemma \ref{L:3.3}. Then $M$ is locally congruent to the image of $\pi\circ L$, where $\pi:H_{1}^{2n+1}(-1)\to CH^{n}(-4)$ is the Hopf fibration and $L$ is one of the following.

\vskip.05in
{\rm (a)} $ L(t, u_{2},\ldots,u_{n})=\( \dfrac{e^{i \theta}\Phi(u_{2},\ldots,u_{n}) }{\sqrt{1-\lambda^2-\varphi^2}}, \dfrac{  e^{i(n-2) \theta}(i\lambda-\varphi)}{\sqrt{1-\lambda^2-\varphi^2}}\), \quad \lambda^2+\varphi^2<1, $
\vskip.05in
\noindent where $\lambda$, $\varphi$ and $\theta$ are functions of $t$ only, satisfying
$$\lam'=(n-3)\lambda\varphi, \quad \varphi'=1-\varphi^2-(n-2)\lambda^2, \quad \theta'=\lambda $$
and $\Phi$ is a Legendre immersion into $H_{1}^{2n-1}(-1)$ whose composition with the Hopf fibration is a minimal $\delta(n-2)$-ideal Lagrangian immersion into $CH^{n-1}(-4)$;

\vskip.05in
{\rm (b)} $L(t, u_{2},\ldots,u_{n})=\( \dfrac{ e^{i(n-2) \theta}(i\lambda-\varphi)}{\sqrt{\lambda^2+\varphi^2-1}},\dfrac{e^{i \theta}\Phi(u_{2},\ldots,u_{n})}{\sqrt{\lambda^2+\varphi^2-1}}\),\;\; \lambda^2+\varphi^2>1, $
\vskip.05in
\noindent where $\lambda$, $\varphi$ and $\theta$ are functions of $t$ only, satisfying
$$ \lambda'=(n-3)\lambda\varphi, \quad \varphi'=1-\varphi^2-(n-2)\lambda^2, \quad \theta'=\lambda, $$
and $\Phi$ is a Legendre immersion into $S^{2n-1}(1)$ whose composition with the Hopf fibration is a minimal $\delta(n-2)$-ideal Lagrangian immersion into $CP^{n-1}(4)$;

\vskip.05in 
{\rm (c)} 
\begin{multline*}
L(t,u_2,\ldots,u_n) = \frac{\cosh^{\frac{2}{n-3}}\left(\frac{n-3}{2}t\right)}{e^{\frac{2i}{n-3} \arctan\left(\tanh(\frac{n-3}{2}t)\right)}} \left[ \left(w + \frac i2 \langle \Phi,\Phi \rangle + i, \Phi, w + \frac i2 \langle \Phi,\Phi \rangle \right) \right. \\
+ \left. \int_0^t \frac{e^{2i \arctan\left(\tanh(\frac{n-3}{2}t)\right)}}{\cosh^{\frac{2}{n-3}} \left(\frac{n-3}{2}t \right)} dt \ (1,0,\ldots,0,1)\right],
\end{multline*}
where $\Phi = \Phi(u_2,\ldots,u_n)$ parametrizes a minimal $\delta(n-2)$-ideal Lagrangian immersion into ${\bf C}^{n-1}$ and $w=w(u_2,\ldots,u_n)$ is the unique solution of the PDE system $ w_{u_{k}}=\< \Phi, i\Phi_{u_{k}}\>$ for $k=2,\ldots,n$. 
\end{proposition}

\section{Classification in case (III) of Lemma \ref{L:3.3}}

In this section we assume that $M$ is a $\delta(2,n-2)$-ideal Lagrangian submanifold of a complex space form $\tilde M^n(4c)$ ($n \geq 5$), whose second fundamental form is given by case (III) of Lemma \ref{L:3.3}. 

\subsection{Proof that $M$ is a warped product} \label{sec5.1}

We define the following two orthogonal distributions on $M$ in terms of the orthonormal frame $\{ E_1, \ldots, E_n \}$:
\begin{equation}\label{3.7i}
\mathcal D_1 = \mathrm{span}\{E_1,E_2\}, \qquad \mathcal D_2 = \mathrm{span}\{E_3,\ldots,E_n\}.
\end{equation}

\begin{lemma}\label{L:5.1} 
Let $M$ be a $\delta(2,n-2)$-ideal Lagrangian submanifold of a complex space form $\tilde M^{n}(4c)$, $n\geq 5$, whose second fundamental form is given by case $\mathrm{(III)}$ of Lemma \ref{L:3.3}. Then
$\mathcal D_2$ is integrable.
\end{lemma}

\begin{proof}
It follows from \e{h_case(III)} that $\<(\nabla_{E_{i}} h)(E_{j},E_1),JE_{1}\>=(2\lambda-\g)\<\nabla_{E_{i}}E_{j},E_{1}\>$ for all $i,j \geq 3$, which, in combination with Codazzi's equation, yields $(2\lambda-\gamma) \<[E_i,E_j],E_1\>=0$. The conditions $\gamma > 0$ and $\gamma > 2n\lambda/3$ imply $2\lambda - \gamma \neq 0$ and hence we obtain
\begin{equation} \label{5.1i}
\<[E_i,E_j],E_1\>=0. 
\end{equation}

It also follows from \e{h_case(III)} that $\<(\nabla_{E_{i}} h)(E_{j},E_1),JE_{2}\>=(\g-(n\!-\!1)\lambda)\<\nabla_{E_{i}}E_{j},E_{2}\>+\mu \<\nabla_{E_{i}}E_{j},E_{1}\>$ for all $i,j \geq 3$, which, in combination with Codazzi's equation and \e{5.1i} gives
\begin{equation} \label{5.1ii}
(\gamma-(n-1)\lambda)\<[E_i,E_j],E_2\>=0. 
\end{equation}

Finally, \e{h_case(III)} implies $\<(\nabla_{E_{i}} h)(E_{j},E_2),JE_{2}\>=(n-2)\mu \<\nabla_{E_{i}}E_{j},E_{2}\>-(n\lambda-\g) \<\nabla_{E_{i}}E_{j},E_{1}\>$ for all $i,j \geq 3$, which, using Codazzi's equation and \e{5.1i}, gives
\begin{equation} \label{5.1iii}
\mu\<[E_i,E_j],E_2\>=0. 
\end{equation}

Combining \e{5.1ii} and \e{5.1iii} with the fact that $\mu \neq 0$ or $\gamma \neq (n-1)\lambda$ implies
\begin{equation} \label{5.1iv}
\<[E_i,E_j],E_2\>=0. 
\end{equation}
Equations \e{5.1i} and \e{5.1iv} together imply that $[E_i,E_j] \in \mathcal D_2$ for all $i,j \geq 3$, which, by Frobenius' theorem, implies that $\mathcal D_2$ is integrable.
\end{proof}

In order to write down the information obtained from the other Codazzi equations, we use the following notations: the one-forms $\omega_j^k$ describing the Levi-Civita connection of $M$ are defined as usual by
\begin{equation}\label{3.7ii}
\omega_j^k(E_i) = \<\nabla_{E_i}E_j,E_k\>
\end{equation}
 
Comparing the $JE_{1}$-, $JE_{2}$- and $JE_{j}$-components ($j = 3,\ldots,n$) of the Codazzi equation $(\nabla_{E_{i}} h)(E_{1},E_1)=(\nabla_{E_{1}} h)(E_{1},E_i)$ ($i = 3,\ldots,n$) gives respectively
\begin{align}\label{5.3}  & E_{i}\g=(\g-2\lambda) \o^{i}_{1}(E_{1}),
\\ &\label{5.4}   (3\g-2n\lambda)\o_{1}^{2} (E_{i})=(\g-(n-1)\lambda) \o^{2}_{i}(E_{1})-\mu \o_{1}^{i}(E_{1}),
\\& \label{5.5} (\g-2\lambda) \o_{1}^{j}(E_{i})=\delta_{ij}(E_{1}\lambda-\mu\o_{1}^{2}(E_{1}))-\sum_{{k=3}}^{n}h_{ij}^{k}\o_{1}^{k}(E_{1})
\end{align} 
for all $i,j \geq 3$. Analogously, $(\nabla_{E_{i}} h)(E_{1},E_2)= (\nabla_{E_{1}} h)(E_{i},E_2)= (\nabla_{E_{2}} h)(E_{1},E_i)$ gives
\begin{align}\label{5.7} 
& (3\g - 2 n \lambda)\o_{1}^{2}(E_{i})= (\g \!-\! 2 \lambda)\o_{1}^{i}(E_{2}) = (\g-(n\!-\!1)\lambda) \o_{i}^{2}(E_{1})-\mu \o_{1}^{i}(E_{1}),
\\& \notag n E_{i}\lambda -(\g-2\lambda)\o_{1}^{i}(E_{1})-n\mu \o_{1}^{2}(E_{i})
=(n-2)\mu \o_{2}^{i}(E_{1})+(n\lambda-\g)\o_{1}^{i}(E_{1})
\\& \quad \label{5.8}  =((n-1)\lambda-\g) \o_{2}^{i}(E_{2})-\mu\o_{1}^{i}(E_{2}),
\\& \notag ((n\!-\!1)\lambda\!-\!\g)\o_{2}^{j}(E_{i})-\mu \o_{1}^{j}(E_{i})=(E_{1}\mu\! +\!\lambda \o_{1}^{2}(E_{1}))\delta_{ij} \! - \sum_{k=3}^{n} h^{k}_{ij}\o_{2}^{k}(E_{1})
\\& \quad \label{5.9} =(E_{2}\lambda -\mu \o_{1}^{2}(E_{2}))\delta_{ij}-\sum_{k=3}^{n} h_{ij}^{k} \o_{1}^{k}(E_{2})
\end{align} 
for $i,j \geq 3$. Finally, it follows from $(\nabla_{E_{i}} h)(E_{2},E_2)= (\nabla_{E_{2}} h)(E_{i},E_2)$ that
\begin{align}\label{5.10} & n E_i\mu+3(n\lambda- \g )\o_{1}^{2}(E_{i})
=(n-2)\mu \o_{2}^{i}(E_{2})-(\g-n\lambda)\o_{1}^{i}(E_{2}),
\\&\label{5.11} (n\lambda-\g)\o_{1}^{j}(E_{i})+(n-2)\mu \o_{2}^{j}(E_{i})
=\delta_{ij}(E_{2}\mu -\lambda \o_{2}^{1}(E_{2}))+\sum_{k=3}^{n} h_{ij}^{k} \o_{2}^{k}(E_{2})\end{align} 
for $i,j \geq 3$.

By changing the orthonormal frame $\{E_3,\ldots,E_n\}$ in $\mathcal{D}_2$ if necessary, we may assume 
that $\o_1^4(E_1)=\cdots=\o_1^n(E_1)=0$ or, equivalently, that the orthogonal projection of $\nabla_{E_1}E_1$ onto $\mathcal D_2$ lies in the direction of $E_3$:
\begin{align}\label{5.12}
\nabla_{E_1}E_1=\o_1^2(E_1)E_2+\omega_1^3(E_1) E_3.
\end{align} 
Thus, we find from $\sum_{i=3}^n \<(\nabla_{E_3}h)(E_i,E_1),JE_i\>=\sum_{i=3}^n \<(\nabla_{E_i}h)(E_3,E_1),JE_i\>$ that
\begin{align}\label{5.14} 
(n-3)(\gamma-2\lambda)(E_3 \lambda-\mu \o_{1}^{2}(E_{3}))=\o_1^3(E_1)\sum_{i,j=3}^n (h^3_{ij})^2.
\end{align}

On the other hand, it follows from \e{5.8} that
\begin{align} \label{5.14bis}
n (E_{3}\lambda -\mu \o_{1}^{2}(E_{3}))
 =(n-2)(\mu \o_{2}^{3}(E_{1})+\lambda\o_{1}^{3}(E_{1})).
\end{align} 
By combining \e{5.14} and \e{5.14bis}, we find
\begin{align}\label{5.15} 
(n^2-5n+6)(\gamma-2\lambda)(\mu \o_{2}^{3}(E_{1})+\lambda\o_{1}^{3}(E_{1}))= n\o_1^3(E_1)\sum_{i,j=3}^n (h^3_{ij})^2.
\end{align}
Similarly, 
$\sum_{i=3}^n \<(\nabla_{E_1}h)(E_i,E_i),JE_3\>=\sum_{i=3}^n \<(\nabla_{E_i}h)(E_1,E_i),JE_3\>$
yields
\begin{align}\label{5.16} 
(n^2-n+2)(\gamma-2\lambda)(\mu \o_{2}^{3}(E_{1})+\lambda\o_{1}^{3}(E_{1}))=
n\o_1^3(E_1)\sum_{i,j=3}^n (h^3_{ij})^2
\end{align}
and $\sum_{i=3}^n \<(\nabla_{E_3}h)(E_i,E_2),JE_i\>=\sum_{i=3}^n \<(\nabla_{E_i}h)(E_2,E_3),JE_i\>$ gives
\begin{align}\label{5.17}  E_3\mu=\lam  \o_{2}^{1}(E_{3})+\frac{1}{n-3}\sum_{i,j=3}^n h^3_{ij}\o_i^2(E_j).\end{align}
From \e{5.15}, \e{5.16} and the properties of $\gamma$, $\lambda$ and $\mu$ in case (III) of Lemma \ref{L:3.3}, we obtain
\begin{align}\label{5.18}  \lambda\o_{1}^{3}(E_{1})+\mu\o_2^3(E_1)=0
\end{align}
and we have either 
(a) $ h^3_{ij}=0$ for all $i,j\geq 3$ and $\o_{1}^{3}(E_{1})\ne 0$  or
(b) $\o_1^3(E_1)=0$.

\vskip.05in
 {\it Case} (a): $ h^3_{ij}=0$ for all $i,j\geq 3$ and $\o_{1}^{3}(E_{1})\ne 0$. Since $\lambda$ and $\mu$ cannot both be zero, \e{5.18} implies that $\mu\ne 0$ and 
\begin{align} \label{5.18bis}
\o_2^3(E_1)=-\frac{\lambda}{\mu}\o_1^3(E_1). 
\end{align}
Equation \e{5.17} and  $h^3_{ij}=0$ imply 
\begin{align}\label{5.19} 
E_3\mu=\lambda\o_2^1(E_3).
\end{align}
Also, it follows from \e{5.4} and \e{5.18bis} that
\begin{align}\label{5.20} &\o_1^2(E_3)=\frac{\lambda\gamma-(n-1)\lambda^2-\mu^2}{(3\gamma-2n\lambda)\mu} \omega^3_1(E_1). \end{align}
By combining \e{5.19} and \e{5.20} we obtain
\begin{align}\label{5.21} E_3\mu=\frac{\lam(\lambda\gamma-(n-1)\lambda^2-\mu^2)}{(2n\lambda-3\gamma)\mu}\omega^3_1(E_1).\end{align}

On the other hand, we find from \e{5.10} that
\begin{align}\label{5.22} & E_3\mu=\frac{3(\g-n\lambda )}{n}\o_{1}^{2}(E_{3})
+\frac{(n-2)\mu}{n} \o_{2}^{3}(E_{2})+\frac{n\lambda-\g}{n}\o_{1}^{3}(E_{2})
.\end{align} 
From  \e{5.20} and the first equality in \e{5.7} we find
\begin{equation}\begin{aligned}\label{5.23} 
\o_1^3(E_2)&\,=\frac{(n-1)\lam^2+\mu^2-\lambda\gamma}{(2\lambda-\g)\mu}\omega^3_1(E_1).\end{aligned}
\end{equation}
Now, \e{5.8}, \e{5.23} and \e{5.18bis} yield
\begin{equation}\begin{aligned}\label{5.25} \o_2^3(E_2)&\,=\frac{(n+3)\lam^2-5\lam\g+\g^2+\mu^2}{(2\lam-\g)((n-1)\lambda-\g)}\omega^3_1(E_1).\end{aligned}\end{equation}
By substituting \e{5.20}, \e{5.23} and \e{5.25} into \e{5.22} we find
\begin{equation}\begin{aligned}\label{5.26} E_3\mu=&\,\frac{3(\g-n\lambda )(\lambda\gamma-(n-1)\lambda^2-\mu^2)}{n(3\gamma-2n\lambda)\mu}\omega^3_1(E_1)
\\&+ \frac{(n-2)\mu((n+3)\lam^2-5\lam\g+\g^2+\mu^2)}{n(2\lam-\g)((n-1)\lambda-\g)}\omega^3_1(E_1)\\&+\frac{(n\lambda-\g)((n-1)\lam^2+\mu^2-\lambda\gamma)}{n(2\lambda-\g)\mu}\omega^3_1(E_1).\end{aligned}\end{equation}
Now, by comparing \e{5.21} and \e{5.26} we find
\begin{align}\notag 
\lam^2((n-1)\lam-\g)^2+\mu^4+\mu^2((\g-3\lam)^2+(2n-7)\lam^2)=0.
\end{align} 
Thus $\mu=0$, which is a contradiction. Hence, case (a) cannot occur.

\vskip.05in

{\it Case} (b): $\o_{1}^{3}(E_{1})=0$. In this case, the choice of $E_3$ we made before becomes arbitrary and thus equation \e{5.18} gives
$\o_{1}^{3}(E_{1})=\mu\o_2^3(E_1)=0$ for arbitrary $E_3$.
Thus, we  have
\begin{align}\label{5.27} 
\o_{1}^{i}(E_{1})=\mu\o_2^i(E_1)=0
\end{align}
for all $i \geq 3$. We can now choose $\{E_3,\ldots,E_n\}$ such that
\begin{align}\label{5.28} 
\nabla_{E_1}E_2=\o_2^1(E_1)E_1+\omega_2^3(E_1) E_3,
\end{align}
i.e., such that $\o_2^4(E_1)=\cdots=\o_2^n(E_1)=0$.
From \e{5.27}  and \e{5.28} we find $\mu\omega_2^3(E_1)=0$. Hence, either
(b.1) $\omega_2^3(E_1)\ne 0$ and $\mu=0$ or 
(b.2) $\omega_2^3(E_1)=0$.

\vskip.05in

{\it Case} (b.1): $\mu=0$ {\it and} $\omega_2^3(E_1)\ne 0$. We find from \e{5.4} and \e{5.7} that
\begin{align}\label{5.30} &0\ne\o_2^3(E_1)=\frac{2n\lam-3\g}{\g-(n-1)\lam}\o_1^2(E_3),
\\ &\label{5.31} \o_1^3(E_2)=\frac{3\g-2n\lam}{\g-2\lam}\o_1^2(E_3).
\end{align}
In particular, \e{5.30} gives
\begin{align}\label{5.32} 
3\g\ne 2n\lam,\;\; \o_1^2(E_3)\ne 0.
\end{align}
Also, from \e{5.4} and \e{5.7}:
\begin{align} \label{5.32.2} 
\o_1^2(E_k) = 0, \ \o_1^k(E_2) = 0
\end{align}
for $k \geq 4$. Since $\mu=0$, \e{5.10} becomes
\begin{align}\label{5.33} & 0=3(n\lambda- \g )\o_{1}^{2}(E_{3})
+(\g-n\lambda)\o_{1}^{3}(E_{2}).\end{align} 
Now, by substituting \e{5.31} into \e{5.33}, we find 
$(\g-n\lam)\lam \o_1^2(E_3)=0$. Since $\o_1^2(E_3) \neq 0$ by \e{5.32} and $\lambda$ and $\mu$ cannot both be zero, this shows that $\lambda\ne 0$ and $\g=n\lam$.

The second fundemental form \e{h_case(III)} reduces to
\begin{align}\label{5.34} 
& h(E_1,E_1) = n\lam JE_1, \ \ h(E_1,E_2) = h(E_2,E_2) = 0, \nonumber \\
& h(E_{1},E_{i})=\lambda JE_{i},\ \ h(E_2,E_i)= 0,\\
&\notag h(E_{i},E_{j})=\delta_{ij}\lambda JE_{1}+\sum_{k=3}^{n} h^{k}_{ij}JE_{k}
\end{align}
for $i, j \geq 3$, where $h_{33}^k+\ldots+h_{nn}^k=0$ for all $k \geq 3$. From \e{5.30} and  \e{5.31} we find
\begin{align}\label{5.36} & \o_{2}^{1}(E_{3})=\frac{\omega_2^3(E_1)}{n},\;\; \o_1^3(E_2)=-\frac{\omega_2^3(E_1)}{n-2}.\end{align} 
Thus \e{5.9}, \e{5.32.2} and \e{5.36} give
\begin{align}\label{5.37} \o_{j}^2(E_{i}) =\frac{E_{2}\lambda}{\lambda}\delta_{ij}+\frac{\omega_2^3(E_1)}{(n-2)\lambda} h_{ij}^{3}. \end{align} 
Now, by using \e{5.17}, \e{5.36} and \e{5.37}, we find 
$$\frac{\omega_2^3(E_1)}{\lambda n} \left( \lambda^2 + \frac{n}{(n-2)(3-n)}\sum_{i,j=3}^n (h^3_{ij})^2 \right) = 0,$$
which is a contradiction. Therefore, this case is again impossible.

\vskip.05in

{\it Case} (b.2): $\omega_2^3(E_1)=0$. From this assumption, we have
\begin{align}\label{5.40} 
\nabla_{E_1}E_1, \nabla_{E_1}E_2\in {\mathcal D}_1.
\end{align} 
It follows  from \e{5.7} that $\o_1^i(E_2)=0$ for $i \geq 3$. Thus  we also have
\begin{align}\label{5.41} \nabla_{E_2}E_1\in {\mathcal D}_1.\end{align}
From the last equation in \e{5.8} we find $((n-1)\lam-\g)\o_2^i(E_2)=0$.  Hence,  either 
(b.2.1) $\g=(n-1)\lambda$ and $\o_2^i(E_2)\ne 0$ for some $i\geq 3$ or 
(b.2.2) $\nabla_{E_2}E_2\in \mathcal D_1$. 

\vskip.05in

 {\it Case} (b.2.i): {\it  $\g=(n-1)\lambda$  and $\o_2^i(E_2)\ne 0$ for some} $i \geq 3$.  In this case, \e{h_case(III)} reduces to
\begin{equation}\begin{aligned}\label{5.42} 
& h(E_1,E_1)=(n-1)\lambda JE_1,
\; h(E_1,E_2)=\lambda JE_2,\; 
\\& h(E_2,E_2)=\lam JE_1 +n \mu JE_2,
\\&h(E_{1},E_{i})=\lambda JE_{i},\; h(E_{2},E_{i})=\mu JE_{i},
\\& h(E_{i},E_{j})=\delta_{ij}(\lambda JE_{1}+\mu J E_{2})+\sum_{k=3}^{n} h^{k}_{ij}JE_{k}
\end{aligned}\end{equation} 
for $i,j \geq 3$ and $h_{33}^k + \ldots h_{nn}^k = 0$ for all $k \geq 3$. We may  assume $\mu\ne 0$, otherwise this case reduces to case (II) of Lemma \ref{L:3.3}.
Without loss of generality, we may assume 
\begin{align}\label{5.43} 
\nabla_{E_2}E_2=\o_2^1(E_2)E_1+\o_2^3(E_2) E_3,
\end{align}
or, equivalently, $\o_2^4(E_2)=\cdots=\o_2^n(E_2)=0$. 
From \e{5.3}--\e{5.5}, \e{5.10} and \e{5.27}, we find
\begin{align}\label{5.45} 
&E_i\g=E_i\lam=\o_1^2(E_i)=E_4\mu=\cdots=E_n\mu=0,
\\ \label{5.46} &\o_i^1(E_j)=-\frac{\delta_{ij}}{(n-3)\lam}(E_1\lam-\mu \o_1^2(E_1)),
\\ \label{5.47} &E_3\mu=\frac{(n-2)\mu \o_2^3(E_2)}{n}, 
\end{align} 
for $i,j \geq 3$.
By applying \e{5.42}--\e{5.47}, we find
\begin{align}\label{5.48} \sum_{i=3}^n \<(\nabla_{E_3}h)(E_i,E_2),JE_i\>=(n-2)E_3\mu = \frac{(n-2)^2 \mu \o_2^3(E_2)}{n}.\end{align}
After long computation we also have
$\sum_{i=3}^n \<(\nabla_{E_2}h)(E_i,E_3),JE_i\>=n\mu \o_2^3(E_2).$
By Codazzi's equation and \e{5.48}, this would imply $\o_2^3(E_2)=0$ or $n=1$, which are both contradictions. Therefore, this case is impossible.

\vskip.05in

{\it Case} (b.2.2): $\nabla_{E_2}E_2\in \mathcal D_1$. 
In this case,  we have
\begin{align}\label{5.50} \o_\alpha^i(E_\beta)=0
\end{align}
for any $\alpha, \beta=1,2$ and $i \geq 3$, i.e., $\mathcal D_1$ is a totally geodesic distribution. From \e{5.4} and \e{5.50} we get \begin{align}\label{5.51} 
(3\g-2n\lam)\o_1^2(E_i)=0
\end{align}
for $i\geq 3$. Consequently,  either 
(b.2.2.1) $3\g=2n\lam$ and $\o_1^2(E_i)\ne 0$ for some $i \geq 3$ or 
(b.2.2.2) $\o_1^2(E_i)=0$ for all $i \geq 3$.

\vskip.05in

{\it Case} (b.2.2.1): $3\g=2n\lam$ and $\o_1^2(E_i)\ne 0$ {\it for some $i \geq 3$}. From \e{5.3}, we obtain $E_i\g=E_i\lam=0$. Hence, we have $\mu=0$ from  \e{5.8}. Also, we find  $(n\lam-\g)\o_1^2(E_i)=0$ from \e{5.10}. Combining this with \e{5.51} yields $\o_1^2(E_i)=0$, which is a contradiction. Consequently, this case cannot occur.

\vskip.05in

{\it Case} (b.2.2.2): $\o_1^2(E_i)=0$ for all $i \geq 3$. 
From \e{5.3}, \e{5.8} and \e{5.10}, we get
\begin{align}\label{5.52} 
E_i\g=E_i\lam=E_i\mu=0
\end{align}
for $i \geq 3$. Recall that $\g\ne 2\lam$, which follows from $\gamma >0$ and $\gamma > 2n\lambda/3$.
We find from \e{5.5}, \e{5.9} and \e{5.11} that
\begin{align}\label{5.53} &\o^1_{i}(E_j)=\frac{\delta_{ij}}{\g-2\lam}(\mu\o_1^2(E_1)-E_1\lam),
\\&\label{5.54} (n-2)\mu\o^2_i(E_j)=(\g-n\lam)\o^1_i(E_j)-
\delta_{ij}(E_2\mu+\lam \o_1^2(E_2)),
\\& \label{5.55} ((n-1)\lam-\g)\o^2_i(E_j)=\mu \o^1_i(E_j)-\delta_{ij}(E_1\mu+\lambda\o_1^2(E_1)),
\\&\label{5.56} ((n-1)\lam-\g)\o^2_i(E_j) =\mu \o^1_i(E_j)-\delta_{ij}(E_2\lam-\mu\o_1^2(E_2))
\end{align}
for $i,j \geq 3$.
Recall that $\mathcal D_1$ is totally geodesic and $\mathcal D_2$ is integrable. It now follows from \e{5.53}--\e{5.56} that leaves of $\mathcal D_2$ are totally umbilical submanifolds of the Lagrangian submanifold $M$. In particular, we may put
\begin{align}\label{5.57} 
\o_i^1(E_j)=p\delta_{ij},\;\; \o_i^2(E_j)=q\delta_{ij}
\end{align}
for some functions $p,q$ and all $i \geq 3$. 
Since $\o_1^2(E_i)=0$ for all $ i \geq 3$, \e{5.57} implies
\begin{align}\label{5.58} 
\nabla_{V}E_1=-pV,\;\; \nabla_V E_2=-qV
\end{align}
for all $V \in \mathcal D_2$. From \e{5.53}--\e{5.57} we get
\begin{equation}\begin{aligned}\label{5.59} & E_1\lam=\mu \o_1^2(E_1)+(2\lam-\g)p,
\\& E_2\lam=\mu\o^2_1(E_2)+\mu p-((n-1)\lam-\g)q,
\\&  E_1\mu = -\lam \o^2_1(E_1)+\mu p-((n-1)\lam-\g) q,
\\& E_2\mu=-\lam \o^2_1(E_2)+(\g-n\lam)p-(n-2)\mu q.
\end{aligned}\end{equation}

By applying \e{5.58}, we find
\begin{equation}\begin{aligned}\label{5.60} 
\<R(E_i,E_j)E_1,E_k\>=(E_jp)\delta_{ik}-(E_ip)\delta_{jk}
\end{aligned}\end{equation}
for all $i,j,k \geq 3$.
On the other hand, it follows from  equation \e{Gauss} of Gauss and \e{h_case(III)} that
\begin{equation}\begin{aligned}\label{5.61} & \<R(E_i,E_j)E_1,E_k\>=0.\end{aligned}\end{equation}
By combining \e{5.60} and \e{5.61} we get $E_jp=0$ for $j \geq 3$. Similarly, we find by computing $\<R(E_i,E_j)E_2,E_k\>$ and using \e{h_case(III)}, \e{5.58} and \e{5.59} that $E_jq=0$. Thus
\begin{align}\label{5.62} 
E_jp=E_jq=0
\end{align}
for all $j \geq 3$.

Now, by applying \e{h_case(III)}, \e{5.50}, \e{5.58}, and the equation of Gauss,
\begin{equation}
\<R(E_\alpha,E_3)E_\beta,E_3\>=-c\delta_{\alpha\beta}+\<h(E_\alpha,E_3),h(E_\beta,E_3)\>
-\<h(E_3,E_3),h(E_\alpha,E_\beta\>,
\end{equation}
where $\alpha,\beta=1,2$, we find 
\begin{equation}\begin{aligned}\label{5.63} 
& E_1p =q \o_1^2(E_1)+p^2-\lam^2+\lam\g+c,
\\& E_2p =q\o^2_1(E_2)+pq+(n\lam-\g)\mu-\lam\mu,
\\&  E_1q = -p \o^2_1(E_1)+pq+(n\lambda-\g)\mu-\lam \mu,
\\& E_2q=-p \o^2_1(E_2)+q^2+n\lam^2+(n-1)\mu^2-\lam\g +c.
\end{aligned}\end{equation}
Also, by applying \e{h_case(III)}, \e{5.50} and \e{5.59}, we find from the equation of Codazzi,
$(\nabla_{E_2}h)(E_1,E_1)= (\nabla_{E_1}h)(E_1,E_2)$,
that 
 \begin{equation}\begin{aligned}\label{5.64} & E_1\g=n(2\lam-\g)p+(2n\lam-3\g) \o_1^2(E_2),
\\& E_2\g=(3\g-2n\lam)\o^2_1(E_1).\end{aligned}\end{equation}
By applying \e{5.59}, \e{5.63} and \e{5.64}, we obtain
 \begin{equation}\begin{aligned}\label{5.65} 
 &E_1(c+\lambda^2+\mu^2+p^2+q^2)=2p(c+\lambda^2+\mu^2+p^2+q^2),
 \\&E_2(c+\lambda^2+\mu^2+p^2+q^2)=2q(c+\lambda^2+\mu^2+p^2+q^2).
 \end{aligned}\end{equation}

If we put $\mathring{H}=p E_1+q E_2$, then  \e{5.58} and \e{5.62} yield $\nabla_{V}\mathring{H}=-(p^2+q^2)V$ for all $V\in \mathcal D_2$, which shows that the mean curvature vector of each leaf of $\mathcal D_2$ is parallel in the normal bundle of this leaf in $M$.
Therefore, $\mathcal D_2$ is a spherical distribution. Consequently, the Lagrangian submanifold $M$ is locally the warped product $M^2\times_f M^{n-2}$ of a leaf $M^2$ of $\mathcal D_1$ and a leaf $M^{n-2}$ of $\mathcal D_2$. Moreover, all the leaves of $\mathcal D_1$ are totally geodesic surfaces in $M$ and all the leaves of $\mathcal D_2$ are spherical submanifolds of $M$.

It is well-known (see, for instance, \cite[page 79]{book}) that the warping function $f$ of the warped product $M^2\times_f M^{n-2}$ satisfies 
\begin{align}\label{5.68} (\nabla_{V}V)^{\mathcal D_1}=-\frac{\mathrm{grad}(f)}{f},\end{align}
for any unit vector field $V \in \mathcal D_2$, where the superscript $\mathcal D_1$ denotes the $\mathcal D_1$-component. Together with \e{5.58}, this implies that
\begin{align} \label{derivatives_of_f} 
E_1(f) = -pf, \ \ E_2(f)=-qf. 
\end{align}
It follows from \e{5.58}, \e{5.59} and \e{derivatives_of_f} that the following two vector fields commute and hence determine coordinates $(x,y)$ on $M^2$:
\begin{equation} \label{coordinates}
\begin{aligned}
& \frac{\partial}{\partial x} = \frac{1}{\lambda^2 + \mu^2} (\lambda E_1 + \mu E_2), \\
& \frac{\partial}{\partial y} = \frac{f^{n-2}}{\lambda^2 + \mu^2} (-\mu E_1 + \lambda E_2).
\end{aligned}
\end{equation}
From \e{derivatives_of_f} and \e{coordinates}, we find that the derivatives of $f$ are
\begin{align} \label{derivatives_of_f_2}
f_x = -\frac{f}{\lambda^2 + \mu^2} (\lambda p + \mu q), \ \
f_y = \frac{f^{n-1}}{\lambda^2 + \mu^2} (\mu p - \lambda q).
\end{align}

Remark that all functions appearing can now be explicitly expressed in terms of $f$. However, since the expressions are complicated and we will not need them to state our final result, we omit them here.

We can summarize this subsection as follows. 

\begin{proposition} \label{P_section_5}
Let $M$ be a $\delta(2,n-2)$-ideal Lagrangian submanifold of a complex space form $\tilde M^{n}(4c)$, $n\geq 5$, whose second fundamental form is given by case $\mathrm{(III)}$ of Lemma \ref{L:3.3}. Then $M$ is locally a warped product $M^2 \times_f M^{n-2}$, where $M^2$ is an integral surface of the distribution $\mathcal D_1 = \mathrm{span}\{E_1,E_2\}$ and $M^{n-2}$ is an integral submanifold of the distribution $\mathcal D_2=\{E_3,\ldots,E_n\}$. Moreover, $M^2$ is totally geodesic in $M$ and $M^{n-2}$ is spherical in $M$, in particular, there exist functions $p$ and $q$ such that $\nabla_VE_1=-pV$ and $\nabla_VE_2=-qV$ for all $V \in \mathcal D_2$. The derivatives of $\gamma$, $\lambda$, $\mu$, $p$, $q$ and $f$ are given by \e{5.52}, \e{5.59}, \e{5.62}, \e{5.63}, \e{5.64} and \e{derivatives_of_f}. Finally, the vector fields \e{coordinates} are coordinate vector fields on $M^2$.
\end{proposition}

In the next three subsections we will classify the $\delta(2,n-2)$-ideal Lagrangian submanifolds whose second fundamental form satisfies case (III) of Lemma \ref{L:3.3} in the ambient spaces $\mathbf C^n$, $CP^n(4)$ and $CH^n(-4)$ respectively.

\subsection{Classification in $\mathbf C^n$}

\begin{proposition}\label{P:6.1} 
Let $L:M \to \mathbf C^n$ ($n \geq 5$) be a $\delta(2,n-2)$-ideal Lagrangian immersion whose second fundamental form is given by case $\mathrm{(III)}$ of Lemma \ref{L:3.3}. 
Then $L$ is locally congruent to
\begin{align}\label{6.2}
L(x,y,u_1,\ldots,u_{n-2})=\big(f(x,y)e^{i x}\Phi(u_1,\ldots,u_{n-2}), z(x,y)\big),\end{align}
where $\Phi$ defines a minimal Legendre immersion in $S^{2n-3}(1)\subset {\bf C}^{n-1}$ and $(fe^{i x},z)$ is a Lagrangian surface in ${\bf C}^2$, where $f$ is determined by 
\begin{align} \label{PDE_for_f_Cn}
\frac{f_{yy}}{f^{n-2}} - (n-2) \frac{f_y^2}{f^{n-1}} + (n-1)f^{n-1} + (n-2) f^{n-3} f_x^2 + f^{n-2} f_{xx} = 0
\end{align}
and $z$ by
\begin{equation} \label{system_for_z_Cn}
\begin{aligned}
& z_x = e^{i(n-1)x} \frac{f_y}{f^{n-2}}, \\
& z_y = e^{i(n-1)x} f^{n-1} \left(i-\frac{f_x}{f}\right).
\end{aligned}
\end{equation}
\end{proposition}

\begin{proof} Let $L:M\to {\bf C}^{n}$ be a $\delta(2,n-2)$-ideal Lagrangian immersion whose second fundamental form is given by case $\mathrm{(III)}$ of Lemma \ref{L:3.3}. It follows from Proposition \ref{P_section_5} that $M$ is locally a warped product  $M^2\times_{f} M^{n-2}$ of a surface $M^2$ and an $(n-2)$-dimensional Riemannian manifold $M^{n-2}$ with warping function $f$ satisfying \e{derivatives_of_f}. Using \e{5.65} and also \e{5.52} and \e{5.62}, which imply that $p$, $q$, $\lambda$ and $\mu$ are constant along $\mathcal D_2$, gives
\begin{align}\label{6.5} 
f=\frac{C}{\sqrt{\lambda^2+\mu^2+p^2+q^2}}
\end{align}
for some positive constant $C$. Now choose coordinates $(x,y)$ as in \e{coordinates} and consider the map $\Phi$ defined by
\begin{equation}
\begin{aligned}\label{6.6} 
\Phi=\frac{e^{-i x} ((p+i \lambda)E_{1}+(q+i \mu)E_{2})}{\sqrt{\lambda^{2}+\mu^{2}+p^{2}+q^{2}}}.
\end{aligned}
\end{equation} 
Remark that $\langle \Phi,\Phi \rangle = 1$. Denote by $D$ the Euclidean connection on $\mathbf C^n$. From \e{3.4}, \e{5.59}, \e{5.63} and \e{5.65} we obtain
\begin{align}\label{6.7}  
D_{E_1}\Phi=D_{E_2}\Phi=0.
\end{align}
Also, by applying \e{3.4}, \e{5.52}, \e{5.58} and \e{5.62}, we find
\begin{align} \label{6.8} 
D_{E_i}\Phi=-e^{-i x}\sqrt{\lam^2 + \mu^2 + p^2 + q^2} E_i
\end{align}
for all $i \geq 3$. This implies that $\Phi$ is an immersion from $M^{n-2}
$ into $S^{2n-1}(1) \subseteq \mathbf C^n$. We will show that the image of $\Phi$ is contained in a linear subspace $\mathbf C^{n-1} \subseteq \mathbf C^n$ and hence in a unit sphere $S^{2n-3}(1) \subseteq \mathbf C^{n-1}$. Therefore, consider for any point $p \in M^{n-2}$ the complex linear subspace $$\mathrm{span}\{\Phi(p), (d\Phi)_p(E_3), \ldots, (d\Phi)_p(E_n)\} \subseteq T_{\Phi(p)}\mathbf C^n.$$
To see that all these subspaces are in fact the same subspace of $\mathbf C^n$, we remark that from \e{3.4}, \e{5.52}, \e{5.57}, \e{5.62} and \e{6.8} 
\begin{equation}
\begin{aligned}\label{6.9} 
D_{E_j}(d\Phi)(E_i) &= D_{E_j} \left( -e^{-i x}\sqrt{\lam^2 + \mu^2 + p^2 + q^2} E_i \right) \\
& =-e^{-i x}\sqrt{\lam^2 + \mu^2 + p^2 + q^2} \, D_{E_j}E_i\\
& =\sum_{k=3}^{n} \(\o_i^k(E_j)+i h_{ij}^k\) (d\Phi)(E_k) -\delta_{ij}(\lam^2 + \mu^2 + p^2 + q^2) \Phi,
\end{aligned}
\end{equation}  
which belongs again to $\mathrm{span}\{\Phi, (d\Phi)(E_3), \ldots, (d\Phi)(E_n)\}$ for any $i,j \geq 3$. We conclude that $\Phi$ is an immersion of $M^{n-2}$ into $S^{2n-3}(1) \subseteq \mathbf C^{n-1} \subseteq \mathbf C^{n}$. Moreover, it follows from the computation above that the second fundamental form of the immersion $\Phi$ coincides with the second fundamental form of $L$ restricted to $M^{n-2}$, which implies that $\Phi$ is a minimal Legendre immersion of $M^{n-2}$ into $S^{2n-3}(1)$. 

Let us put
\begin{align}&  \label{6.10}
\Psi = \frac{(-q+i \mu)E_1 +(p-i \lam)E_2}{e^{i(n-1)x} \sqrt{\lambda^2 + \mu^2 + p^2 + q^2}}.
\end{align}
Then $\Psi$ is orthogonal to $\Phi$ and 
\begin{equation}\begin{aligned}\label{6.11} 
& D_{E_j}\! \(L+\frac{e^{i x}}{\sqrt{\lam^2\!+\!\mu^2\!+\!p^2\!+\!q^2}}\Phi\)=0 \mbox{ if } j \geq 3, \\
& D_{E_1}\! \(L+\frac{e^{i x}}{\sqrt{\lam^2\!+\!\mu^2\!+\!p^2\!+\!q^2}}\Phi\)=\frac{-e^{i(n-1)x}(q+i \mu)}{\sqrt{\lam^2\!+\!\mu^2\!+\!p^2\!+\!q^2}}\Psi, \\
&D_{E_2}\! \(L+\frac{e^{i x}}{\sqrt{\lam^2\!+\!\mu^2\!+\!p^2\!+\!q^2}}\Phi\)=\frac{e^{i(n-1)x}(p+i \lam)}{\sqrt{\lam^2\!+\!\mu^2\!+\!p^2\!+\!q^2}}\Psi.
\end{aligned}\end{equation}
Moreover, since $D_{E_A}\Psi = 0$ for all $A = 1,\ldots,n$, we can assume that, after a suitable isometry of the ambient space, $\Psi = (0,\ldots,0,1)$. Consequently, $L$ takes the form
\begin{align}\label{6.12}L(x,y,u_1,\ldots,u_{n-2})=\(-\frac 1C f(x,y)e^{i x}\Phi(u_1,\ldots,u_{n-2}), z(x,y)\),
\end{align}
where $z$ is a complex valued function whose derivatives are essentially computed in \e{6.11}. By using \e{coordinates}, we obtain that $z$ satisfies 
\begin{align*}
& z_x = \frac 1C e^{i(n-1)x} \frac{f_y}{f^{n-2}}, \\
& z_y = \frac 1C e^{i(n-1)x} f^{n-1} \left(i-\frac{f_x}{f}\right).
\end{align*} 
The compatibility condition for this system is precisely \e{PDE_for_f_Cn}.

Note that everything is invariant under a rescaling of $f$, if, at the same time, we rescale the $y$-coordinate corresponding to \e{coordinates}. Hence, we may assume $C=1$. Moreover, after an isometry of $\mathbf C^n$ we may omit the minus signs in the first $n-1$ components of $L$, obtaining \eqref{6.2}. 

Since $L:M_1^2\times_f M_2^{n-2}\to {\bf C}^{n}$ is Lagrangian, it follows from \e{6.12} that $\Phi$ is a Legendre minimal immersion in $S^{2n-3}(1)\subset {\bf C}^{n-1}$. Note that $(fe^{ix},z)$ is a Lagrangian surface in ${\bf C}^2$.

The converse can be verified by direct long computation.
\end{proof}

\begin{example}
Let us construct an explicit example in dimension $n=5$ by assuming that $f_y=0$. In that case, the general solution of \e{PDE_for_f_Cn} is given by
\begin{equation}
f(x) = c_1 (\cos(4x-c_2))^{1/4}.
\end{equation}
It then follows from the system \e{system_for_z_Cn} that $z$ is independent of $x$ and
\begin{equation}
z(y) = ic_1^4 e^{ic_2} y.
\end{equation}
Note that $c_1$ cannot be zero since $(f(x)e^{ix},z(y))$ has to be an immersion. It is a product of two plane curves and thus a Lagrangian surface in ${\bf C}^2$.
\end{example}

\subsection{Classification in $CP^{n}$}

\begin{proposition}\label{P:7.1} 
Let $L:M\to CP^{n}(4)$ ($n \geq 5$) be a $\delta(2,n-2)$-ideal Lagrangian immersion whose second fundamental form is given by case $\mathrm{(III)}$ of Lemma \ref{L:3.3}. Then the horizontal lift $\tilde L: M\to S^{2n+1}(1)\subseteq {\bf C}^{n+1}$ of $L$ is given by
\begin{equation} \label{7.1}
\tilde L(x,y,u_1,\ldots,u_{n-2})= e^{ix} f(x,y) \Phi(u_1,\ldots,u_{n-2}) + e^{i(n-1)x} \sqrt{1-f(x,y)^2} \, \Theta_2(x,y) ,
\end{equation}
where, with respect to a suitable orthogonal decomposition $\mathbf C^{n+1} = \mathbf C^{n-1} \oplus \mathbf C^2$, the map $\Phi: M^{n-2} \to S^{2n-3}(1) \subseteq \mathbf C^{n-1}$ is a minimal Legendre immersion, the warping function $f$ is determined by
\begin{multline}
(1-f^2) f^{2n-3} f_{xx} + (1-f^2) f f_{yy} + ( (n-2)(1-f^2)+2f^2 ) f^{2n-4} f_x^2 \\ - ( (n-2)(1-f^2)-2f^2 ) f_y^2 + ( (n-1)(1-f^2)+2f^2 ) f^{2n-2} = 0
\end{multline}
and $\Theta_2: M^2 \to S^3(1) \subseteq \mathbf C^2$ is a solution of the system
\begin{equation}
\begin{aligned}
& (\Theta_1)_x = \frac{1}{1-f^2} \( i f^2 \Theta_1 - \frac{f_y}{f^{n-2}} \Theta_2 \), \\
& (\Theta_1)_y = \frac{1}{1-f^2} f^{n-2} (f_x + if) \Theta_2, \\
& (\Theta_2)_x = \frac{1}{1-f^2} \( \frac{f_y}{f^{n-2}} \Theta_1 + i((1-n)(1-f^2)-f^2) \Theta_2 \), \\
& (\Theta_2)_y = \frac{1}{1-f^2} f^{n-2} (-f_x + if) \Theta_1.
\end{aligned}
\end{equation}
\end{proposition}

\begin{proof} 

From \eqref{5.65} and \eqref{derivatives_of_f}, we obtain that
\begin{equation}
f = \frac{C}{\sqrt{1+\lambda^2+\mu^2+p^2+q^2}}.
\end{equation}
Now define the following three maps:
\begin{equation} \label{7.5}
\begin{aligned}
& \Phi = e^{-ix} \frac{\tilde L - (p+i\lambda) E_1 - (q+i\mu) E_2}{\sqrt{1+\lambda^2+\mu^2+p^2+q^2}}, \\
& \Theta_1 = e^{-i(n-1)x} \frac{(-q+i\mu) E_1 + (p-i\lambda)E_2}{\sqrt{1+\lambda^2+\mu^2+p^2+q^2}}, \\
& \Theta_2 = e^{-i(n-1)x} \frac{(\lambda^2+\mu^2+p^2+q^2)\tilde L + (p+i\lambda) E_1 + (q+i\mu) E_2}{\sqrt{\lambda^2+\mu^2+p^2+q^2}\sqrt{1+\lambda^2+\mu^2+p^2+q^2}}, \\
\end{aligned}
\end{equation}
where $x$ is the coordinate on $M^2$ defined in \eqref{coordinates}. Then $\langle \Phi,\Phi \rangle = 1$ and, denoting by $D$ the Euclidean connection on $\mathbf C^{n+1}$, one also has $D_{E_1}\Phi = D_{E_2}\Phi = 0$ and
\begin{equation} \label{7.6}
D_{E_j}\Phi = e^{-ix} \sqrt{1+\lambda^2+\mu^2+p^2+q^2} E_j
\end{equation}
for $j \geq 3$. This implies that $\Phi$ is an immersion from $M^{n-2}$ into $S^{2n+1}(1) \subseteq \mathbf C^{n+1}$. 

We will now show that the image of $\Phi$ is actually contained in a linear subspace $\mathbf C^{n-1}$ of $\mathbf C^{n+1}$ and, since it has length one, in $S^{2n-3}(1) \subseteq \mathbf C^{n-1}$. Therefore, consider for any point $p \in M^{n-2}$ the complex linear subspace $$\mathrm{span}\{\Phi(p), (d\Phi)_p(E_3), \ldots, (d\Phi)_p(E_n)\} \subseteq T_{\Phi(p)}\mathbf C^{n+1}.$$
To see that all these subspaces are in fact the same subspace of $\mathbf C^{n+1}$, we use \eqref{7.6} and the fact that
\begin{equation*}
\begin{aligned}
D_{E_j}(d\Phi)(E_k) = & \ D_{E_j} \( e^{-ix} \sqrt{1+\lambda^2+\mu^2+p^2+q^2} E_k \) \\
= & \ e^{-ix} \sqrt{1+\lambda^2+\mu^2+p^2+q^2} \Big( \delta_{jk} ((p+i\lambda)E_1 + (q+i\mu)E_2 - \tilde L) \\
& + \sum_{\ell=3}^n (\omega_k^{\ell}(E_i)+ih_{jk}^{\ell})E_{\ell} \Big) \\
= & \ - \delta_{jk} (1+\lambda^2+\mu^2+p^2+q^2) \Phi + \sum_{\ell=3}^n (\omega_k^{\ell}(E_i)+ih_{jk}^{\ell}) (d\Phi)(E_{\ell})
\end{aligned}
\end{equation*}
belongs again to $\mathrm{span}\{\Phi, (d\Phi)(E_3), \ldots, (d\Phi)(E_n)\}$ for any $j,k \geq 3$. We conclude that $\Phi$ is an immersion of $M^{n-2}$ into $S^{2n-3}(1) \subseteq \mathbf C^{n-1} \subseteq \mathbf C^{n+1}$. Moreover, it follows from the computation above that the second fundamental form of the immersion $\Phi$ coincides with the second fundamental form of $\tilde L$ restricted to $M^{n-2}$, which implies that $\Phi$ is a minimal Legendre immersion of $M^{n-2}$ into $S^{2n-3}(1)$.

It is clear that $\Theta_1$ and $\Theta_2$ take values in the orthogonal complement $\mathbf C^2$ of $\mathbf C^{n-1}$ in $\mathbf C^{n+1}$. Moreover, $D_{E_j}\Theta_1 = D_{E_j}\Theta_2 = 0$ for $j \geq 3$ and the derivatives of $\Theta_1$ and $\Theta_2$ in the directions of $E_1$ and $E_2$ are linear combinations of $\Theta_1$ and $\Theta_2$. With respect to the coordinates $(x,y)$ introduced in \eqref{coordinates}, we have
\begin{equation}
\begin{aligned}
& (\Theta_1)_x = \frac{1}{C^2-f^2} \( i f^2 \Theta_1 - \frac{C f_y}{f^{n-2}} \Theta_2 \), \\
& (\Theta_1)_y = \frac{C}{C^2-f^2} f^{n-2} (f_x + if) \Theta_2, \\
& (\Theta_2)_x = \frac{1}{C^2-f^2} \( \frac{C f_y}{f^{n-2}} \Theta_1 + i((1-n)(C^2-f^2)-f^2) \Theta_2 \), \\
& (\Theta_2)_y = \frac{C}{C^2-f^2} f^{n-2} (-f_x + if) \Theta_1.
\end{aligned}
\end{equation}
The integrability condition for this system is 
\begin{multline*}
(C^2-f^2) f^{2n-3} f_{xx} + (C^2-f^2) f f_{yy} + ( (n-2)(C^2-f^2)+2f^2 ) f^{2n-4} f_x^2 \\ - ( (n-2)(C^2-f^2)-2f^2 ) f_y^2 + ( (n-1)(C^2-f^2)+2f^2 ) f^{2n-2} = 0.
\end{multline*}
Since everything is invariant under a rescaling of $f$, if we rescale the $y$-coordinate accordingly, cfr. \eqref{coordinates}, we may assume $C=1$. This yields the equations for $f$, $\Theta_1$ and $\Theta_2$ given in the proposition. The expression for $\tilde L$ follows directly from \eqref{7.5}.

The converse can be verified by a long but straightforward computation.
\end{proof}

\subsection{Classification in $CH^{n}$}

\begin{proposition}\label{P:8.1} Let $L:M\to CH^{n}(-4)$ ($n \geq 5$) be a $\delta(2,n-2)$-ideal Lagrangian immersion whose second fundamental form is given by case $\mathrm{(III)}$ of Lemma \ref{L:3.3}. Then the horizontal lift $\tilde L: M\to H^{2n+1}_1(-1)\subseteq {\bf C}^{n+1}_1$ of $L$ is given by one of the following.

\vskip.05in 
{\rm (a)} With respect to a suitable orthogonal decomposition $\mathbf C^{n+1}_1 = \mathbf C^{n-1} \oplus \mathbf C^2_1$, 
\begin{equation}\label{8.1}
\tilde L(x,y,u_1,\ldots,u_{n-2}) =  -e^{ix} f(x,y) \Phi(u_1,\ldots,u_{n-2}) + e^{i(n-1)x} \sqrt{1+f(x,y)^2} \, \Theta_2(x,y) ,
\end{equation}
where $\Phi: M^{n-2} \to S^{2n-3}(1) \subseteq \mathbf C^{n-1}$ is a minimal Legendre immersion, the warping function $f$ is determined by
\begin{multline} \label{8.2}
(1+f^2) f^{2n-3} f_{xx} + (1+f^2) f f_{yy} + ((n-2)(1+f^2)-2f^2) f^{2n-4} f_x^2 \\ - ((n-2)(1+f^2)+2f^2) f_y^2 + ((n-1)(1+f^2)-2f^2) f^{2n-2} = 0
\end{multline}
and $\Theta_2: M^2 \to H^3_1(-1) \subseteq \mathbf C^2_1$ is a solution of the system 
\begin{equation} \label{8.3}
\begin{aligned}
& (\Theta_1)_x = \frac{1}{1+f^2} \(-i f^2 \Theta_1 - \frac{f_y}{f^{n-2}} \Theta_2 \), \\
& (\Theta_1)_y = \frac{1}{1+f^2} f^{n-2} (f_x + if) \Theta_2, \\
& (\Theta_2)_x = \frac{1}{1+f^2} \( -\frac{f_y}{f^{n-2}} \Theta_1 - i((n-1)(1+f^2)-f^2) \Theta_2 \), \\
& (\Theta_2)_y = \frac{1}{1+f^2} f^{n-2} (f_x - if) \Theta_1.
\end{aligned}
\end{equation}
 
\vskip.05in 
{\rm (b)} With respect to a suitable orthogonal decomposition $\mathbf C^{n+1}_1 = \mathbf C^{n-1}_1 \oplus \mathbf C^2$, 
\begin{equation}\label{8.4}
\tilde L(x,y,u_1,\ldots,u_{n-2}) =  e^{ix} f(x,y) \Phi(u_1,\ldots,u_{n-2}) - e^{i(n-1)x} \sqrt{f(x,y)^2-1} \, \Theta_2(x,y) ,
\end{equation}
where $\Phi: M^{n-2} \to H^{2n-3}_1(-1) \subseteq \mathbf C^{n-1}_1$ is a minimal Legendre immersion, the warping function $f$ is determined by
\begin{multline} \label{8.5}
(f^2-1) f^{2n-3} f_{xx} + (f^2-1) f f_{yy} + ((n-2)(f^2-1)-2f^2) f^{2n-4} f_x^2 \\ - ((n-2)(f^2-1)+2f^2) f_y^2 + ((n-1)(f^2-1)-2f^2) f^{2n-2} = 0
\end{multline}
and $\Theta_2: M^2 \to S^3(1) \subseteq \mathbf C^2_1$ is a solution of the system 
\begin{equation} \label{8.6}
\begin{aligned}
& (\Theta_1)_x = \frac{1}{f^2-1} \(-i f^2 \Theta_1 - \frac{f_y}{f^{n-2}} \Theta_2 \), \\
& (\Theta_1)_y = \frac{1}{f^2-1} f^{n-2} (f_x + if) \Theta_2, \\
& (\Theta_2)_x = \frac{1}{f^2-1} \( \frac{f_y}{f^{n-2}} \Theta_1 + i((n-1)(f^2-1)-f^2) \Theta_2 \), \\
& (\Theta_2)_y = \frac{1}{f^2-1} f^{n-2} (-f_x + if) \Theta_1.
\end{aligned}
\end{equation}

\vskip.05in 
{\rm (c)} With respect to the local coordinates $(x,y)$ on $M^2$ introduced above and local coordinates $(u_3,\ldots,u_n)$ on $M^{n-2}$,
\begin{equation} \label{8.6o}
\tilde L = f e^{ix}(u+iv+1,u+iv,G,\bar F), 
\end{equation}
where the warping function $f:M^2 \to \mathbf R$ is a solution of
\begin{equation} \label{8.6a} 
f^{2n-3}f_{xx} + ff_{yy} + (n-4)f^{2n-4}f_x^2 - nf_y^2 +(n-3)f^{2n-2} = 0, 
\end{equation}
$F: M^2 \to \mathbf C$ is determined by
\begin{equation} \label{8.6b} 
F_x = -e^{-i(n-3)x}\frac{f_y}{f^n}, \qquad F_y = e^{-i(n-3)x} f^{n-4} (f_x+if), 
\end{equation}
$G: M^{n-2} \to \mathbf C^{n-2}$ is a minimal Lagrangian immersion,
$u: M \to \mathbf R$ is given by
\begin{equation} \label{8.6c} 
u = \frac 12 (\langle G,G \rangle + |F|^2 - 1) + \frac{1}{2f^2}
\end{equation}
and $v: M \to \mathbf R$ is determined by
\begin{equation} \label{8.6d} 
\begin{aligned}
& v_x = -\frac{1}{f^2} - \frac{f_y}{f^n} \Im(e^{i(n-3)x}F), \\
& v_y = -f^{n-3} \Re(e^{i(n-3)x}F) + f^{n-4}f_x \Im(e^{i(n-3)x}F), \\
& v_{u_k} = \langle D_{\frac{\partial}{\partial u_k}} G,iG \rangle.
\end{aligned}
\end{equation}
\end{proposition}

\begin{proof}
We divide the proof into three cases.

\vskip.05in
{\it Case} (1): $\lam^2+\mu^2+p^2+q^2>1$.
It follows from \eqref{5.65} and \eqref{derivatives_of_f} that
\begin{equation}
f = \frac{C}{\sqrt{\lambda^2+\mu^2+p^2+q^2-1}}
\end{equation}
for some real constant $C > 0$. Now consider the maps
\begin{equation}
\begin{aligned}\label{8.7} 
& \Phi  =e^{-i x}\frac{\tilde L+(p+i \lambda)E_{1}+(q+i \mu)E_{2}}{\sqrt{\lambda^{2}+\mu^{2}+p^{2}+q^{2}-1}}, \\
& \Theta_{1} = e^{-i(n-1)x} \frac{(q - i \mu)E_1 -(p-i\lambda)E_2}{\sqrt{\lambda^{2}+\mu^2+p^{2}+q^2}},\; \\
& \Theta_{2} = e^{-i(n-1)x} \frac{(\lambda^{2}+\mu^2+p^{2}+q^2)\tilde L +(p+i \lam)E_{1}+(q+i \mu)E_2}{\sqrt{\lam^2+\mu^{2}+p^2+q^{2}}\sqrt{\lam^2+\mu^{2}+p^2+q^{2}-1}}.
\end{aligned}
\end{equation}
Then $\langle \Phi,\Phi \rangle = \langle \Theta_1,\Theta_1 \rangle = 1$ and $\langle \Theta_2,\Theta_2 \rangle = -1$. Continuing in the same way as in the proof of Proposition \ref{P:7.1}, we obtain case (a).

\vskip.05in
{\it Case} (2): $\lam^2+\mu^2+p^2+q^2<1$. 
It follows from \eqref{5.65} and \eqref{derivatives_of_f} that
\begin{equation}
f = \frac{C}{\sqrt{1-\lambda^2-\mu^2-p^2-q^2}}
\end{equation}
for some real constant $C > 0$. Now consider the maps
\begin{equation}
\begin{aligned}\label{8.7bis} 
& \Phi  =e^{-i x}\frac{\tilde L+(p+i \lambda)E_{1}+(q+i \mu)E_{2}}{\sqrt{1-\lambda^2-\mu^2-p^2-q^2}}, \\
& \Theta_{1} = e^{-i(n-1)x} \frac{(q - i \mu)E_1 -(p-i\lambda)E_2}{\sqrt{\lambda^{2}+\mu^2+p^{2}+q^2}},\; \\
& \Theta_{2} = e^{-i(n-1)x} \frac{(\lambda^{2}+\mu^2+p^{2}+q^2)\tilde L +(p+i \lam)E_{1}+(q+i \mu)E_2}{\sqrt{\lam^2+\mu^{2}+p^2+q^{2}}\sqrt{1-\lambda^2-\mu^2-p^2-q^2}}.
\end{aligned}
\end{equation}
Then $\langle \Phi,\Phi \rangle = -1$ and $\langle \Theta_1,\Theta_1 \rangle = \langle \Theta_2,\Theta_2 \rangle = 1$. Continuing in the same way as in the proof of Proposition \ref{P:7.1}, we obtain case (b).

\vskip.05in
{\it Case} (3): $\lam^2+\mu^2+p^2+q^2=1$. Let us put
\begin{align}\label{8.16}  &\Phi = e^{-ix}f (\tilde L+(p+i \lam)E_1+(q+i \mu)E_2 ),
\\ & \label{8.17} \Theta=e^{-i(n-2)x}\left((q-i \mu)E_1-(p-i \lam)E_2 \right),\end{align}
where $x$ is the coordinate on $M^2$ defined by \eqref{coordinates}. Then we have
\begin{align}&\label{8.18} \! \<\Phi,\Phi\>=\<\Phi,\Theta\>=0,\;\; \<\Theta,\Theta\>=1, 
\\&\label{8.19} D_{E_1}\Phi=D_{E_{2}}\Phi=D_{E_{i}}\Phi=D_{E_i}\Theta=0,\;\; i=3,\ldots,n,
\\&\label{8.20} D_{E_1}\Theta=e^{-i(n-3)x}\frac{q-i \mu}{f}\Phi,
\\&\label{8.21} D_{E_2}\Theta=-e^{-i(n-3)x}\frac{p-i \lam}{f}\Phi.
\end{align}
It follows from \e{8.18} and \e{8.19} that $\Phi$ is a constant light-like vector. Moreover, from \e{8.19}, \e{8.20} and \e{8.21}, together with \eqref{coordinates} and \eqref{derivatives_of_f_2}, we obtain that $\Theta$ can be seen as a map from $M^2$ satisfying
\begin{equation}\begin{aligned}\label{8.22} 
& \Theta_x = -e^{-i(n-3)x} \frac{f_y}{f^n} \Phi, \\
& \Theta_y = e^{-i(n-3)x} f^{n-4} (f_x+if) \Phi.
\end{aligned}\end{equation}
This implies that 
\begin{align}\label{8.23}
\Theta = c_0 + F \Phi,
\end{align} for some space-like unit vector $c_0$ perpendicular to $\Phi$ and a function $F: M^2 \to \mathbf C$ satisfying \eqref{8.6b}. Remark that \e{8.6a} is the integrability condition for the system \e{8.6b}.

From \e{8.16}  we obtain
\begin{equation} \label{8.25}  
\tilde  L = \frac{e^{ix}}{f}\Phi + \Psi,
\end{equation}
where 
\begin{equation} \label{8.25a}
\Psi=-(p+i \lam)E_1-(q+i \mu)E_2.
\end{equation}
Since $\<\right.\!\tilde L,E_1\! \left.\>=\<\right.\!\tilde L,i E_1\! \left.\>=\<\right.\!\tilde L,E_2\! \left.\>=\<\right.\! \tilde L,i E_2\! \left.\>=0$, we find from \e{8.25} and \e{8.25a}
\begin{equation}\begin{aligned} \label{8.26}  
& \<e^{ix}\Phi,E_1\>=fp, & & \<e^{ix}\Phi,E_2\>=fq, \\ 
& \<e^{ix}\Phi,i E_1\>=f\lam, & & \<e^{ix}\Phi,i E_1\>=f\mu,
\end{aligned}\end{equation}
or, equivalently,
\begin{equation}\begin{aligned} \label{8.27} 
& \< \Phi,E_1 \> = f (p\cos x + \lam\sin x),
&& \< \Phi,iE_1 \> = f (\lam\cos x - p\sin x), \\
& \< \Phi,E_2 \> = f (q\cos x + \mu\sin x),
&& \< \Phi,iE_2 \> = f (\mu\cos x - q\sin x).
\end{aligned}\end{equation}
It then follows from \e{8.25a} and \e{8.27} that
\begin{equation} \label{8.28}
\< \Psi,\Phi \> = -f \cos x, \qquad \< \Psi,i\Phi \> = -f \sin x.
\end{equation}
We obtain from \e{8.17} and \e{8.25a} that $\langle\Psi,\Theta\rangle = 0$ and together with \e{8.23} and \e{8.28} this implies
\begin{equation} \label{8.29} 
\<\Psi,c_0\> = f \, \Re (\bar F e^{ix}), \qquad \<\Psi,ic_0\> = f \, \Im (\bar F e^{ix}).\end{equation}
Without loss of generality, we may choose 
\begin{equation}\label{8.30} 
\Phi=(1,1,0,\ldots,0), \qquad  c_0=(0,0,\ldots,0,1).
\end{equation}  
It then follows from \e{8.28}--\e{8.30} that $\Psi=(\Psi_2+fe^{ix},\Psi_2,\Psi_3,\ldots,\Psi_n,f\bar Fe^{ix})$ for some functions $\Psi_2,\ldots,\Psi_n:M\to\mathbf C$. Now define real valued functions $\alpha$, $\beta$ and complex valued functions $G_3, \ldots, G_n$ by 
\begin{equation*} \label{8.31}
\Psi_2 = fe^{ix}(\alpha+i\beta), \qquad (\Psi_3,\ldots,\Psi_n)=fe^{ix}(G_3,\ldots,G_n).
\end{equation*}
Then, from \e{8.25},
\begin{equation} \label{8.32}
\tilde L = fe^{ix} \left( \frac{1}{f^2} + \alpha + i\beta + 1, \frac{1}{f^2} + \alpha + i\beta, G_3, \ldots, G_n, \bar F \right)
\end{equation}
and the conditon $\langle \tilde L, \tilde L \rangle = -1$ yields
\begin{equation} \label{8.33}
\alpha = \frac 12 \left( \langle G,G \rangle + |F|^2 - 1 \right) - \frac{1}{2f^2},
\end{equation}
where $\langle G,G \rangle$ denotes the square of the length of $G=(G_3,\ldots,G_n)$ in $\mathbf C^{n-2}$. By putting $u = \alpha+1/f^2$ and $v = \beta$, we obtain the desired expression for $\tilde L$ and \e{8.6c}. 

Let us now check that $G$ does not depend on $x$ and $y$ and is actually a Lagrangian immersion of $M^{n-2}$ into $\mathbf C^{n-2}$. Using \e{coordinates}, \e{8.16} and \e{8.17}, we can express $\tilde L_x$ and $\tilde L_y$ as linear combinations of $\tilde L$, $\Phi$ and $\Theta$. Since $\Phi=(1,1,,0,\ldots,0,0)$ and $\Theta=(F,F,0,\ldots,0,1)$, it follows from these expressions that $\tilde L_3,\ldots,\tilde L_n$ satisfy
\begin{equation} \label{8.34}
\begin{aligned}
& (\tilde L_j)_x = -\frac{1}{\lambda^2+\mu^2}(\lambda p + \mu q - i(\lambda^2+\mu^2)) \tilde L_j, \\
& (\tilde L_j)_y = \frac{f^{n-2}}{\lambda^2+\mu^2}(\mu p - \lambda q) \tilde L_j.
\end{aligned}
\end{equation}
By using that $\tilde L_j = f e^{ix} G_j$ for $j=3,\ldots,n$ and \eqref{derivatives_of_f_2}, we obtain from \e{8.34} that $(G_j)_x=(G_j)_y=0$. To show that $G:M^{n-2} \to \mathbf{C}^{n-2}$ is Lagrangian, it suffices to check that $\langle G_{u_j},iG_{u_k}\rangle = 0$ for all $j,k=1,\ldots,n-2$. A straightforward computation shows that $\langle \tilde L_{u_j},i \tilde L_{u_k} \rangle = f^2 \langle G_{u_j}, i G_{u_k}\rangle$ and since $\tilde L$ is Legendrian, we have $\langle \tilde L_{u_j},i \tilde L_{u_k} \rangle = 0$ so that we obtain the result.

Finally, we check that $v$ satisfies the system \e{8.6d}. Since $\tilde L$ is horizontal, we have $\langle \tilde L_x, i\tilde L \rangle = \langle \tilde L_y, i\tilde L \rangle = \langle \tilde L_{u_j}, i\tilde L \rangle = 0$ for all $j=1,\ldots,n-2$. A straightforward computation, using \e{8.6o}, \e{8.6b} and \e{8.6c} then gives the result.

The converse can be verified by long but straightforward computation.
\end{proof}

\section{Main theorems}

Finally, we summarize our results from above as the three main theorems.

\begin{theorem} \label{T:9.1} 
Let $M$ be a Lagrangian submanifold of the complex Euclidean space ${\bf C}^{n}$ with $n\geq 5$. Then we have the inequality
 \begin{equation*} 
\delta(2,n\hskip-.01in-\hskip-.01in 2) \leq \text{$ \frac{n^2(n-2)}{4(n-1)} $} H^2 \end{equation*}
at every point. Assume that $M$ is non-minimal. Then the equality sign in the above inequality holds identically, i.e., $M$ is $\delta(2,n-2)$-ideal, if and only if $M$ is locally congruent to the image of one of the following two immersions:

\vskip.05in 
{\rm (a)} 
\begin{equation*}
L(x, u_{2},\ldots,u_{n})=\frac{e^{i\theta(x)}}{\varphi(x)+ix}\Phi( u_{2},\ldots,u_{n}),  
\end{equation*}
with 
\begin{equation*} 
\theta(x) = \frac{n-1}{2-n} \arcsin \( cx^{\frac{n-2}{n-3}} \), \qquad\varphi(x) = \sqrt{\frac{1}{c^2 \, x^{\frac{2}{n-3}}}-x^2}, 
\end{equation*}
where $c$ is a positive constant and $\Phi$ is a minimal Legendre submani\-fold of $S^{2n-1}(1) \subseteq \mathbf C^n$ which is mapped to a $\delta(n-2)$-ideal minimal Lagrangian submanifold of $CP^{n-1}(4)$ by the Hopf fibration;
 
\vskip.05in   
{\rm (b)}
\begin{align*} L(x,y,u_1,\ldots,u_{n-2})=\big(f(x,y)e^{i x}\Phi(u_1,\ldots,u_{n-2}), z(x,y)\big),\end{align*}
where $\Phi$ defines a minimal Legendre immersion in $S^{2n-3}(1)\subset {\bf C}^{n-1}$ and $(fe^{i x},z)$ is a Lagrangian surface in ${\bf C}^2$, where $f$ is determined by 
\begin{align*} 
\frac{f_{yy}}{f^{n-2}} - (n-2) \frac{f_y^2}{f^{n-1}} + (n-1)f^{n-1} + (n-2) f^{n-3} f_x^2 + f^{n-2} f_{xx} = 0
\end{align*}
and $z$ by
\begin{equation*}
\begin{aligned}
& z_x = e^{i(n-1)x} \frac{f_y}{f^{n-2}}, \\
& z_y = e^{i(n-1)x} f^{n-1} \left(i-\frac{f_x}{f}\right).
\end{aligned}
\end{equation*}
\end{theorem}

\begin{remark}
As pointed out in Remark \ref{RemarkMinimal}, if $M$ is a minimal $\delta(2,n-2)$-ideal Lagrangian submanifold of $\mathbf C^n$ and the bases given in Lemma \ref{L:3.2} can be pasted together to form an orthonormal frame, then $M$ is either $\delta(2)$-ideal, $\delta(n-2)$-ideal or $\delta(2,k)$-ideal for some $k$ satisfying $2 \leq k < n-2$ or it is given by
\begin{align*}  
L(x,y,u_1,\ldots,u_{n-2})=\big(L_1(x,y),L_2(u_1,\ldots,u_{n-2})\big), 
\end{align*}
where $L_1$ is a minimal $\delta(2)$-ideal Lagrangian immersion into $\mathbf C^2$ and $L_2$ is a minimal $\delta(n-2)$-ideal Lagrangian immersion into ${\bf C}^{n-2}$.
\end{remark}

\begin{theorem} \label{T:9.2} 
Let $M$ be a Lagrangian submanifold of the complex projective space $CP^{n}(4)$, with $n\geq 5$. Then we have the inequality
\begin{equation*} 
\delta(2,n\hskip-.01in-\hskip-.01in 2) \leq \text{$ \frac{n^2(n-2)}{4(n-1)} $} H^2+2(n-2) \end{equation*}
at every point. Assume that $M$ is non-minimal. Then the equality sign in the above inequality holds identically, i.e., $M$ is $\delta(2,n-2)$-ideal, if and only if $M$ is locally congruent to the image of $L = \pi \circ \tilde L$, where $\pi: S^{2n+1}(1) \to CP^n(4)$ is the Hopf fibration and $\tilde L$ is one of the following two immersions:

\vskip.05in 
{\rm (a)} 
\begin{equation*}
\tilde L(x, u_{2},\ldots,u_{n})=\(\frac{e^{i \theta}\Phi(u_{2},\ldots,u_{n})}{\sqrt{1+\lambda^2+\varphi^2}}, \frac{ e^{i(n-2)\theta}(i\lambda-\varphi)}{\sqrt{1+\lambda^2+\varphi^2}}\),  
\end{equation*}
where $\theta$, $\varphi$ and $\lambda$ are functions of $x$ only, satisfying
\begin{equation*}  
\lambda'=(n-3)\lambda\varphi, \quad \varphi'=-1-\varphi^2-(n-2)\lambda^2, \quad \theta'=\lambda,
\end{equation*}
and $\Phi$ is a Legendre immersion into $S^{2n-1}(1)$ whose image under the Hopf fibration is minimal $\delta(n-2)$-ideal Lagrangian in $CP^{n-1}(4)$;

\vskip.05in   
{\rm (b)} 
\begin{equation*}  
\tilde L(x,y,u_1,\ldots,u_{n-2}) \\ =  e^{ix} f(x,y) \Phi(u_1,\ldots,u_{n-2}) + e^{i(n-1)x} \sqrt{1-f(x,y)^2} \, \Theta_2(x,y) ,
\end{equation*}
where, with respect to a suitable orthogonal decomposition $\mathbf C^{n+1} = \mathbf C^{n-1} \oplus \mathbf C^2$, the map $\Phi$ is a minimal Legendre immersion into $S^{2n-3}(1) \subseteq \mathbf C^{n-1}$ , the real function $f$ is determined by
\begin{multline*}
(1-f^2) f^{2n-3} f_{xx} + (1-f^2) f f_{yy} + ( (n-2)(1-f^2)+2f^2 ) f^{2n-4} f_x^2 \\ - ( (n-2)(1-f^2)-2f^2 ) f_y^2 + ( (n-1)(1-f^2)+2f^2 ) f^{2n-2} = 0
\end{multline*}
and $\Theta_2$ is map into $S^3(1) \subseteq \mathbf C^2$, which is a solution of
\begin{equation*}
\begin{aligned}
& (\Theta_1)_x = \frac{1}{1-f^2} \( i f^2 \Theta_1 - \frac{f_y}{f^{n-2}} \Theta_2 \), \\
& (\Theta_1)_y = \frac{1}{1-f^2} f^{n-2} (f_x + if) \Theta_2, \\
& (\Theta_2)_x = \frac{1}{1-f^2} \( \frac{f_y}{f^{n-2}} \Theta_1 + i((1-n)(1-f^2)-f^2) \Theta_2 \), \\
& (\Theta_2)_y = \frac{1}{1-f^2} f^{n-2} (-f_x + if) \Theta_1.
\end{aligned}
\end{equation*}
\end{theorem}

\begin{remark}
As pointed out in Remark \ref{RemarkMinimal}, if $M$ is a minimal $\delta(2,n-2)$-ideal Lagrangian submanifold of $CP^n(4)$ and the bases given in Lemma \ref{L:3.2} can be pasted together to form an orthonormal frame, then $M$ is either $\delta(2)$-ideal, $\delta(n-2)$-ideal or $\delta(2,k)$-ideal for some $k$ satisfying $2 \leq k < n-2$.
\end{remark}

\begin{theorem} \label{T:9.3} 
Let $L:M\to CH^n(-4)$ be a Lagrangian submanifold of the complex hyperbolic space $CH^{n}(-4)$ with $n\geq 5$. Then we have the inequality
\begin{equation*}
\delta(2,n\hskip-.01in-\hskip-.01in 2) \leq \text{$ \frac{n^2(n-2)}{4(n-1)} $} H^2-2(n-2) \end{equation*}
at every point. Assume that $M$ is non-minimal. Then the equality sign in the above inequality holds identically, i.e., $M$ is $\delta(2,n-2)$-ideal, if and only if $M$ is locally congruent to the image of $L = \pi \circ \tilde L$, where $\pi: H^{2n+1}_1 \to CH^n(-4)$ is the Hopf fibration and $\tilde L$ is one of the following six immersions:

\vskip.05in 
{\rm (a)}
$$ \tilde L(x, u_{2},\ldots,u_{n})=\( \dfrac{e^{i \theta}\Phi(u_{2},\ldots,u_{n}) }{\sqrt{1-\lambda^2-\varphi^2}}, \dfrac{  e^{i(n-2) \theta}(i\lambda-\varphi)}{\sqrt{1-\lambda^2-\varphi^2}}\), \quad \lambda^2+\varphi^2<1, $$
where $\lambda$, $\varphi$ and $\theta$ are functions of $x$ only, satisfying
$$\lam'=(n-3)\lambda\varphi, \quad \varphi'=1-\varphi^2-(n-2)\lambda^2, \quad \theta'=\lambda, $$
and $\Phi$ is a Legendre immersion into $H_{1}^{2n-1}(-1)$ whose image under the Hopf fibration is minimal $\delta(n-2)$-ideal Lagrangian in $CH^{n-1}(-4)$;

\vskip.05in
{\rm (b)}
$$ \tilde L(x, u_{2},\ldots,u_{n})=\( \dfrac{ e^{i(n-2) \theta}(i\lambda-\varphi)}{\sqrt{\lambda^2+\varphi^2-1}},\dfrac{e^{i \theta}\Phi(u_{2},\ldots,u_{n})}{\sqrt{\lambda^2+\varphi^2-1}}\),\;\; \lambda^2+\varphi^2>1, $$
where $\lambda$, $\varphi$ and $\theta$ are functions of $x$ only, satisfying
$$ \lambda'=(n-3)\lambda\varphi, \quad \varphi'=1-\varphi^2-(n-2)\lambda^2, \quad \theta'=\lambda, $$
and $\Phi$ is a Legendre immersion into $S^{2n-1}(1)$ whose image under the Hopf fibration is minimal $\delta(n-2)$-ideal Lagrangian in $CP^{n-1}(4)$;

\vskip.05in 
{\rm (c)} 
\begin{multline*}
\tilde L(x,u_2,\ldots,u_n) = \frac{\cosh^{\frac{2}{n-3}}\left(\frac{n-3}{2}x\right)}{e^{\frac{2i}{n-3} \arctan\left(\tanh(\frac{n-3}{2}x)\right)}} \left[ \left(w + \frac i2 \langle \Phi,\Phi \rangle + i, \Phi, w + \frac i2 \langle \Phi,\Phi \rangle \right) \right. \\
+ \left. \int_0^x \frac{e^{2i \arctan\left(\tanh(\frac{n-3}{2}t)\right)}}{\cosh^{\frac{2}{n-3}} \left(\frac{n-3}{2}t \right)} dt \ (1,0,\ldots,0,1)\right],
\end{multline*}
where $\Phi$ is a minimal $\delta(n-2)$-ideal Lagrangian immersion into ${\bf C}^{n-1}$ and $w$ is the unique solution of the PDE system $ w_{u_{k}}=\< \Phi, i\Phi_{u_{k}}\>$ for $k=2,\ldots,n$;

\vskip.05in 
{\rm (d)}
\begin{equation*}
\tilde L(x,y,u_1,\ldots,u_{n-2}) \\ =  -e^{ix} f(x,y) \Phi(u_1,\ldots,u_{n-2}) + e^{i(n-1)x} \sqrt{1+f(x,y)^2} \, \Theta_2(x,y) ,
\end{equation*}
where, with respect to a suitable orthogonal decomposition $\mathbf C^{n+1}_1 = \mathbf C^{n-1} \oplus \mathbf C^2_1$, the map $\Phi$ is a minimal Legendre immersion into $S^{2n-3}(1) \subseteq \mathbf C^{n-1}$, the real function $f$ is determined by
\begin{multline*} 
(1+f^2) f^{2n-3} f_{xx} + (1+f^2) f f_{yy} + ((n-2)(1+f^2)-2f^2) f^{2n-4} f_x^2 \\ - ((n-2)(1+f^2)+2f^2) f_y^2 + ((n-1)(1+f^2)-2f^2) f^{2n-2} = 0
\end{multline*}
and $\Theta_2$ is a map into $H^3_1(-1) \subseteq \mathbf C^2_1$, which is a solution of
\begin{equation*} 
\begin{aligned}
& (\Theta_1)_x = \frac{1}{1+f^2} \(-i f^2 \Theta_1 - \frac{f_y}{f^{n-2}} \Theta_2 \), \\
& (\Theta_1)_y = \frac{1}{1+f^2} f^{n-2} (f_x + if) \Theta_2, \\
& (\Theta_2)_x = \frac{1}{1+f^2} \( -\frac{f_y}{f^{n-2}} \Theta_1 - i((n-1)(1+f^2)-f^2) \Theta_2 \), \\
& (\Theta_2)_y = \frac{1}{1+f^2} f^{n-2} (f_x - if) \Theta_1;
\end{aligned}
\end{equation*}
 
\vskip.05in 
{\rm (e)}
\begin{equation*}
\tilde L(x,y,u_1,\ldots,u_{n-2}) \\ =  e^{ix} f(x,y) \Phi(u_1,\ldots,u_{n-2}) - e^{i(n-1)x} \sqrt{f(x,y)^2-1} \, \Theta_2(x,y),
\end{equation*}
where, with respect to a suitable orthogonal decomposition $\mathbf C^{n+1}_1 = \mathbf C^{n-1}_1 \oplus \mathbf C^2$, the map $\Phi$ is a minimal Legendre immersion into $H^{2n-3}_1(-1) \subseteq \mathbf C^{n-1}_1$, the real function $f$ is determined by
\begin{multline*} 
(f^2-1) f^{2n-3} f_{xx} + (f^2-1) f f_{yy} + ((n-2)(f^2-1)-2f^2) f^{2n-4} f_x^2 \\ - ((n-2)(f^2-1)+2f^2) f_y^2 + ((n-1)(f^2-1)-2f^2) f^{2n-2} = 0
\end{multline*}
and $\Theta_2$ is a map into $S^3(1) \subseteq \mathbf C^2_1$, which is a solution of 
\begin{equation*} 
\begin{aligned}
& (\Theta_1)_x = \frac{1}{f^2-1} \(-i f^2 \Theta_1 - \frac{f_y}{f^{n-2}} \Theta_2 \), \\
& (\Theta_1)_y = \frac{1}{f^2-1} f^{n-2} (f_x + if) \Theta_2, \\
& (\Theta_2)_x = \frac{1}{f^2-1} \( \frac{f_y}{f^{n-2}} \Theta_1 + i((n-1)(f^2-1)-f^2) \Theta_2 \), \\
& (\Theta_2)_y = \frac{1}{f^2-1} f^{n-2} (-f_x + if) \Theta_1;
\end{aligned}
\end{equation*}

\vskip.05in 
{\rm (f)}
\begin{multline*}
\tilde L(x,y,u_1,\ldots,u_{n-2}) = f(x,y) e^{ix}(u(x,y,u_1,\ldots,u_{n-2})+iv(x,y,u_1,\ldots,u_{n-2})+1, \\ u(x,y,u_1,\ldots,u_{n-2})+iv(x,y,u_1,\ldots,u_{n-2}),G(u_1,\ldots,u_{n-2}), F(x,y)), 
\end{multline*}
where the real function $f$ is determined by
\begin{equation*}  
f^{2n-3}f_{xx} + ff_{yy} + (n-4)f^{2n-4}f_x^2 - nf_y^2 +(n-3)f^{2n-2} = 0
\end{equation*}
and the complex function $F$ by
\begin{equation*}  
F_x = -e^{i(n-3)x}\frac{f_y}{f^n}, \qquad F_y = e^{i(n-3)x} f^{n-4} (f_x-if), 
\end{equation*}
$G$ is a minimal Lagrangian immersion into $\mathbf C^{n-2}$, the real function 
$u$ is given by
\begin{equation*}   
u = \frac 12 (\langle G,G \rangle + |F|^2 - 1) + \frac{1}{2f^2}
\end{equation*}
and the real function $v$ is determined by
\begin{equation*}  
\begin{aligned}
& v_x = -\frac{1}{f^2} - \frac{f_y}{f^n} \Im(e^{i(n-3)x}F), \\
& v_y = -f^{n-3} \Re(e^{i(n-3)x}F) + f^{n-4}f_x \Im(e^{i(n-3)x}\bar F), \\
& v_{u_k} = \langle D_{\frac{\partial}{\partial u_k}} G,iG \rangle.
\end{aligned}
\end{equation*}
\end{theorem}

\begin{remark}
As pointed out in Remark \ref{RemarkMinimal}, if $M$ is a minimal $\delta(2,n-2)$-ideal Lagrangian submanifold of $CH^n(-4)$ and the bases given in Lemma \ref{L:3.2} can be pasted together to form an orthonormal frame, then $M$ is either $\delta(2)$-ideal, $\delta(n-2)$-ideal or $\delta(2,k)$-ideal for some $k$ satisfying $2 \leq k < n-2$.
\end{remark}


\begin{thebibliography}{12}

\bibitem{bo} J. Bolton, F.  Dillen,  J.  Fastenakels and L. Vrancken, A best possible inequality for curvature-like tensor fields, {\it  Math. Inequal. Appl.} {\bf 12} (2009),  663--681.

\bibitem{BMV} J. Bolton, C. Rodriguez Montealegre and L. Vrancken, {Characterizing warped product Lagrangian immersions in complex projective space}, {\it Proc. Edinb. Math. Soc.} \textbf{51} (2008), 1--14.

\bibitem{BV} J. Bolton and L. Vrancken, {Lagrangian submanifolds attaining equality in the improved Chen's inequality}, {\it Bull. Belg. Math. Soc. Simon Stevin} \textbf{14} (2007), 311--315.

\bibitem{c2} B.-Y. Chen, Some pinching and classification theorems for minimal submanifolds, {\it Arch. Math.} {\bf 60} (1993), 568--578.

\bibitem{israel} B.-Y. Chen, Interaction of Legendre curves and Lagrangian submanifolds, {\it Israel J. Math.} {\bf 99} (1997), 69--108.

\bibitem{tohoku} B.-Y. Chen, Complex extensors and Lagrangian submanifolds in complex Euclidean spaces, {\it Tohoku Math. J.} {\bf 49} (1997), 277--297.

\bibitem{c00a} B.-Y. Chen, Some new obstruction to minimal and Lagrangian isometric immersions, {\it Japan. J. Math.} {\bf 26} (2000), 105--127.

 \bibitem{c00b} B.-Y. Chen,  {Ideal Lagrangian immersions in complex space forms}, {Math. Proc. Cambridge Philos. Soc.} {\bf 128} (2000),  511--533.

\bibitem{book} B.-Y. Chen, Pseudo-Riemannian Geometry, $\delta$-invariants and Applications, World Scientific, Hackensack, New Jersey, 2011.

\bibitem{CD11} B.-Y. Chen and F. Dillen, Optimal general  inequalities for Lagrangian submanifolds in complex space forms, {\it J. Math. Anal. Appl.}  {\bf 379} (2011), 229--239. 

\bibitem{cdvv96}  B.-Y. Chen, F. Dillen, L. Verstraelen and L. Vrancken, {An exotic totally real minimal immersion of $S^3$ in $CP^3$ and its characterization},  {Proc. Roy. Soc. Edinburgh Sec. A Math.} {\bf 126} (1996), 153--165.

\bibitem{cdv}  B.-Y. Chen, F. Dillen and L. Vrancken,  Lagrangian submanifolds in complex space forms attaining equality in a basic inequality, {\it J. Math. Anal. Appl.} {\bf 386} (2012), 139--152.

\bibitem{cdvv}  B.-Y. Chen, F. Dillen, J. Van der Veken and L. Vrancken, Curvature inequalities for Lagrangian submanifolds: the final solution, Differ. Geom. Appl. {\bf 31} (2013), 808--819. 

\bibitem{CO} B.-Y. Chen and K. Ogiue, {On totally real submanifolds,} {\it Trans. Amer. Math. Soc.} {\bf 193} (1974), 257--266.

\bibitem{cvv}  B.-Y. Chen, J. Van der Veken and L. Vrancken, Lagrangian submanifolds with prescribed second fundamental form, in: Pure and Applied Differential Geometry PADGE 2012 (2013), 91--98. 

\bibitem{cpw} B.-Y. Chen, A. Prieto-Mart\'{i}n and X. Wang, Lagrangian submanifolds  in complex space forms  satisfying an improved equality involving $\delta(2,2)$, Publ. Math. Debrecen {\bf 82} (2013), 193--217.

\bibitem{sdsv} F. Dillen, C. Scharlach, K. Schoels and L. Vrancken, {Special Lagrangian 4-folds with $SO(2)\rtimes S_{3}$-symmetry in complex space forms}, Taiwanese J. Math. {\bf 19} (2015), 759--792.

\bibitem{rec}  H. Reckziegel,  Horizontal lifts of isometric immersions into the bundle space of a pseudo-Riemannian submersion, {\it Lecture Notes in Mathematics,} {\bf 1156} (1985), 264--279.


\end{thebibliography}
\end{document}